\documentclass[a4paper,11pt]{amsart}

\usepackage{amsthm,amsmath,amssymb,stmaryrd,enumerate,euscript,mathrsfs,a4,array}
\usepackage{geometry}

\usepackage[foot]{amsaddr}
\usepackage{quiver}
\usepackage{pictexwd,dcpic,graphicx,cite}
\usepackage[export]{adjustbox}
\usepackage{subcaption}
\usepackage[latin1]{inputenc} 
\usepackage{indentfirst}
\usepackage{calligra}
\usepackage{cleveref}
\usepackage{graphicx}
\usepackage{amssymb}
\usepackage{fancyhdr}
\usepackage{amsmath}
\usepackage{latexsym}
\usepackage{amsthm}
\usepackage[all]{xy} 
\usepackage{amscd}
\usepackage{tikz-cd}
\usepackage{tikz}
\usetikzlibrary{decorations.pathmorphing}
\usetikzlibrary{arrows.meta}
\numberwithin{equation}{section} 

\graphicspath{ {images/} }

\newtheorem{theorem}{Theorem}[section]
\newtheorem*{theorem*}{Theorem}
\newtheorem{prop}[theorem]{Proposition}

\newtheorem{lemma}[theorem]{Lemma}
\newtheorem{cor}[theorem]{Corollary}

\theoremstyle{definition} 
\newtheorem{defn}[theorem]{Definition}

\theoremstyle{remark}
\newtheorem{rmk}[theorem]{Remark}
\newtheorem{eg}[theorem]{Example}

\newtheorem*{remark*}{Remark}
\newtheorem*{remarks*}{Remarks}
\newtheorem*{notation*}{Notation}
\newtheorem*{convention*}{Convention}

\def\black{\color{black}}

\def\yellow{\color{black}}

\newcommand{\op}{\mathrm{op}}

\newcommand{\cat}{\mathbb}

\newcommand{\catC}{\cat{C}}

\newcommand{\catE}{\cat{E}}
\newcommand{\catX}{\cat{X}}
\newcommand{\catY}{\cat{Y}}
\newcommand{\catZ}{\cat{Z}}

\newcommand{\catK}{\mathcal{K}}
\newcommand{\catL}{\mathcal{L}}

\newcommand{\tcat}{\mathbf}

\newcommand{\Cat}{\tcat{Cat}}
\newcommand{\vcat}{\mathcal{V}\text{-}\tcat{Cat}}
\newcommand{\Set}{\mathbf{Set}}
\newcommand{\Fin}{\mathbf{Fin}}

\newcommand{\Mnd}{\mathbf{Mnd}}
\newcommand{\Rel}{\mathbf{Rmd}}

\newcommand{\Kl}{\textrm{Kl}}
\newcommand{\Lift}{\mathbf{Lift}}

\newcommand{\RLift}{\mathbf{LiftR}}

\newcommand{\Int}{\mathbf{Cat}}
\newcommand{\Mon}{\mathbf{Mon}}
\newcommand{\Pos}{\mathbf{Pos}}

\newcommand{\algfont}{\mathrm}

\newcommand{\Salg}{\ms\text{-}\algfont{Alg}}

\newcommand{\Soalg}{\mso\text{-}\algfont{Alg}}

\newcommand{\mso}{S_0}

\newcommand{\ms}{S}

\newcommand{\REM}{\mathrm{Mod}}
\newcommand{\RRM}{\mathrm{Mod}}

\newcommand{\til}{~}
\newcommand{\compS}{(S,\,S_0)}

\title[]{Distributive laws for relative monads}
\author[]{Gabriele Lobbia}
\address[]{University of Leeds, School of Mathematics, Leeds, United Kingdom. 
ORCID: 0000-0002-2732-0317.}
\email[]{lobbia@math.muni.cz}

\begin{document}

\begin{abstract}
We introduce the notion of a distributive law between a relative monad and a monad. We call this a relative distributive law and define it in any 2-category $\catK$. In order to do that, we introduce the 2-category of relative monads in a 2-category $\catK$ with relative monad morphisms and relative monad transformations as 1- and 2-cells, respectively. We relate our definition to the 2-category of monads in $\catK$ defined by Street. \yellow Using this perspective, \black we prove two Beck-type theorems regarding relative distributive laws. We also describe what does it mean to have Eilenberg--Moore and Kleisli objects in this context and give examples in the 2-category of locally small categories. 
\end{abstract}

\maketitle

\section*{Acknowledgements} The author is very grateful to Martin Hyland for the helpful conversation about the definition of extension to Kleisli in this particular case. This paper also owes a lot to Nicola Gambino's suggestions and feedback. Discussions with Francesco Gallinaro and Giovanni Sold\`a helped the author dealing with some examples. This research is part of the author's PhD project, supported by an EPSRC Scholarship. The second version of this paper was improved also thanks to suggestions and comments by Nathanael Arkor and John Bourke. 

\section*{Introduction} 

\subsection*{Context}
Monads are very helpful tools both in mathematics (see \cite{BarrM:toptt}) and in computer science (see \cite{MoggiE:comp-and-monads}). They were first introduced as endofunctors $S\colon\catC\to\catC$ with natural transformations $m\colon  S^2\rightarrow S$ and $s\colon 1_{\catC}\rightarrow S$ acting as multiplication and unit. Then, Manes \cite[Definition\til3.2]{ManesE:alg-th} introduced the equivalent notion of a Kleisli triple, which relies on a mapping of objects $S\colon\textrm{Ob}(\catC)\to\textrm{Ob}(\catC)$, an extension operator sending any map $f\colon X\rightarrow SY$ to one of the type $f^\dagger\colon SX\rightarrow SY$ and a family of maps $s_X\colon X\rightarrow SX$. In recent years, monads with this description have been called \textit{no--iteration monads} or \emph{(left) extension systems}, and they have been studied in \cite{HernER:No-it-distr, ManesE:mnd-comp, MarmVaz:no-it-distr, MarmWood:ext-syst}. 

This description of monads leads to a generalisation, known as \emph{relative monads} \cite[Definition~2.1]{AltChap:mnd-no-end}. These are monad-like structures on a base functor $I\colon\catC_0\rightarrow\catC$, i.e.\ for any $X\in\catC_0$ an object $SX\in\catC$, for any $X,\,Y\in\catC_0$ a extension operator $(-)^\dagger_S\colon\catC(IX,\,SY)\rightarrow\catC(SX,\,SY)$ and a unit $s_X\colon IX\rightarrow SX$ satisfying unital and associativity laws.


In monad theory an important notion is the one of a distributive law \cite{BeckJ:disl, MarmRos:basic-distr} of a monad $T$ over another monad $S$, i.e.\ a natural transformation $d\colon ST\rightarrow TS$ satisfying four compatibility axioms. In \cite{BeckJ:disl} Beck proved that a distributive law $d\colon ST\rightarrow TS$ is equivalent to a lifting of $T$ to $S$-algebras. It is well known, and often attributed to Beck as well, that a distributive law $d\colon ST\rightarrow TS$ is also equivalent to an extension of $S$ to the Kleisli category of $T$.

In \cite[Proposition~3.5]{MarmRos:basic-distr} we find a characterization of distributive laws $d\colon ST\to TS$ in terms of $S$-algebras $\alpha\colon STS\to TS$ with some properties. This description is extended in \cite[Theorem~6.2]{MarmWood:ext-syst} to extension systems. Then \cite{HernER:No-it-distr} provides the definition of a distributive law of a right extension system with respect to a left extension system (also called a \emph{no-iteration distributive law} or a \emph{distributive law in extensive form}), where a right extension system is the dual notion of an extension system. Finally, mixed distributive laws (between a monad and a comonad) have been studied in terms of extension systems in \cite{MarmVaz:no-it-distr}.

The main aim of this paper is to develop further the theory of distributive laws by introducing the notion of a distributive law between a relative monad $T$ and a monad $S$, which we call a \emph{relative distributive laws} (\Cref{defn:rel-distr}). In particular we prove a counterpart of Beck's equivalence for relative distributive laws (\Cref{thm:rel-beck}). 

We take a 2-categorical approach to the subject, inspired by the formal theory of monads \cite{LackS:fortm, StreetR:fortm}, using also ideas from \cite{MarmolejoF:dislp}.  
Let us briefly recall how distributive laws can be treated using this point of view. First, for a 2-category $\catK$, one introduces the 2-category $\Mnd(\catK)$ of monads, monad morphisms and monad 2-cells \cite{LackS:fortm, StreetR:fortm}. Then, one introduces the notions of an (indexed) left module and left module morphism, uses them to introduce a 2-category $\Lift(\catK)$ of monads, liftings of maps to left modules and lifting of 2-cells to left modules (approach used for pseudomonads in \cite{MarmolejoF:dislp}), and proves that $\Mnd(\catK)$ and $\Lift(\catK)$ are 2-isomorphic. Once this is done, everything follows formally. First, one gets an equivalence between distributive laws (which are monads in $\Mnd(\catK)$) and liftings of monads to left modules (which are monads in $\Lift(\catK)$). Secondly, by duality, one obtains a 2-isomorphism $\Mnd(\catK^\op)^\op$ with $\Lift(\catK^\op)^\op$, which leads to the corresponding result on the equivalence between distributive laws and extensions of monads to right modules. Since representability of left and right modules corresponds to existence of Eilenberg--Moore and Kleisli objects, respectively, in that case one gets a version of Beck's theorem. Importantly, in the equivalence between distributive laws $d\colon ST\to TS$ and liftings of $T$ to the category of Eilenberg--Moore algebras of $S$, one considers $S$ as a part of an object of $\Mnd(\catK)$ and $T$ as part of a monad morphism, while in the equivalence between distributive laws $d\colon ST\to TS$ and extensions of $S$ to the Kleisli category of $T$, one considers $T$ as part of an object of $\Mnd(\catK)$ and $S$ as part of a  monad morphism. The equivalence between all these notions is possible because of the aforementioned duality and because $\Mnd(\catK^\op)^\op$ has the same objects as $\Mnd(\catK)$.

We will introduce a 2-category $\Rel(\catK)$ of relative monads, relative monad morphisms and relative monad transformations in $\catK$. This 2-category generalises the one of no-iteration monads introduced in \cite{HernER:No-it-distr}. Importantly, $\Rel(\catK)$ is more closely related to $\Mnd(\catK^\op)^\op$ rather than to $\Mnd(\catK)$. Indeed, it contains $\Mnd(\catK^\op)^\op$ as a full sub-2-category (\Cref{prop:mnd-subcat-rel}). \yellow This is motivated \black by the fact that relative monads are particularly suited to study Kleisli categories. Then, we extend some results of \cite{AltChap:mnd-no-end, FioreM:relpsm} to our setting, proving them for any 2-category. Using this point of view, we introduce a notion of distributive law of a relative monad $T$ on a monad (Definition~\ref{defn:comp-mnd}), which we call \emph{relative distributive law} (Section~\ref{sec:rel-distr-law}). We then show that it is equivalent to an object of $\Mnd(\Rel(\catK))$.

The first difference we find between our work and the formal theory of monads is that the objects of $\Rel(\catK^\op)$ are not the same as those of $\Rel(\catK)$. The issue is that the notion of operator (Definition~\ref{defn:operator}) that is involved in the definition of a relative monad does not dualise, i.e.\ an operator in $\catK^\op$ is not an operator in $\catK$. For this reason, the duality available for monads fails and we need to consider separately left and right modules. In each case, we are able to prove some, but not all, counterparts of some of the results valid in the classical case. Remarkably, the combination of these results still allows us to obtain a version of Beck's theorem (Theorem~\ref{thm:rel-beck}).

Using left modules for a relative monad we are able to find a relative adjunction (Theorem~\ref{lemma:rel-EM-to-rel-mnd}). In particular, thanks to this result we can prove that if a 2-category has relative Eilenberg--Moore objects, then any relative monad is induced by a relative adjunction (Theorem~\ref{thm:EM-obj}). On the other hand, we do not have a correspondence between relative monad morphisms and liftings of morphisms to left modules. Nevertheless, we get an equivalence between relative distributive laws and liftings of relative monads to left modules (Theorem~\ref{thm:rel-distr-eq-lift}).  

Considering right modules we do not get a relative adjunction (see Remark~\ref{rmk:rel-adj-vs-yon}). Instead, we use them to get a correspondence between relative monad morphisms and liftings of morphisms to right modules (Proposition~\ref{prop:lift-iff-rel-1-cells}). In particular, we can define a 2-category $\RLift(\catK)$ of liftings to right modules and prove that it is 2-isomorphic to $\Rel(\catK)$. Thus, we get an equivalence between  relative distributive laws and liftings of monads to right modules of a relative monad (Theorem~\ref{thm:rel-distr-eq-kl}), for which we use an argument similar to the one in \cite{StreetR:fortm}.

A further motivation for this work is to provide a first step towards the definition of a notion of a pseudodistributive law between a relative pseudomonad \cite{FioreM:relpsm} and a 2-monad \cite{BlackR:2-mnd}. Part of a Beck-like theorem has already been translated in this setting \cite[Theorem\til6.3]{FioreM:relpsm} without defining the notion of a pseudodistributive law. With this definition, it will be possible to interpret the results in \cite[Section\til7]{FioreM:relpsm} with a \emph{relative pseudodistributive law} of the presheaf relative pseudomonad on the 2-monad for free monoidal categories, or for symmetric monoidal categories etc. 

\yellow In this paper we work within a 2-category, but it is worth mentioning Arkor's PhD thesis \cite[Chapter~5]{Arkor:PhDTh} where relative monads in a \emph{proarrow equipment} are introduced. \black

\subsection*{Outline of the Paper} In Section~\ref{sect:prel} we introduce our notation and the definition of operator, which generalises the notion of a family of maps $\catC(Fx,Gy)\to\catC(F'x,G'y)$ natural in $x$ and $y$. Section~\ref{sec:rel-in-K} uses operators to translate some results for relative monads in $\Cat$ to any 2-category $\catK$. Then, in Section\til\ref{sec:Rel(K)}, we define explicitly the 2-category $\Rel(\catK)$. In Section\til\ref{sec:rel-alg} we define algebras for a relative monad and use them to describe when a relative monad is induced by a relative adjunction. In Section\til\ref{sec:rel-distr-law} we define a relative distributive law and then prove the first Beck-type theorem. Section\til\ref{sec:rrm-kl-ext} is devoted to the 2-isomorphism between $\RLift(\catK)$ and $\Rel(\catK)$ and the second Beck-type theorem. We conclude the paper with some examples. 

\section{Preliminary Definitions}
\label{sect:prel}

Throughout this chapter, for a 2-category $\catK$, we use letters $X,\,Y,\,Z$ ... to denote 0-cells, $F\colon X\rightarrow Y$, $G\colon Y\rightarrow Z$ ... for 1-cells and $f\colon F\rightarrow G$, $\alpha\colon F\rightarrow F'$ ... for 2-cells. Regarding compositions, we will write $G\circ F$ or $GF$ for composition of 1-cells. For 2-cells we denote with $\beta\circ\alpha\colon GF\rightarrow G'F'$ or juxtaposition for horizontal composition and $f'\cdot f\colon F\rightarrow H$ for vertical composition. We will denote with $(A,B)\colon O\rightarrow X;Y$ spans in $\catK$ and with $(F,G)\colon Y;X\rightarrow Z$ cospans, i.e.\ diagrams as below.
\begin{center}
\begin{tikzpicture}
\node (T) at (1,1) {$O$};
\node (T') at (2,0) {$Y$};
\node (ST') at (0,0) {$X$};
\path[->] 
(T) edge node[scale=.7] [above, xshift=-0.3cm, yshift=-0.1cm] (A) {$A$} (ST')
	edge node[scale=.7] [above, xshift=0.3cm, yshift=-0.1cm] {$B$} (T');

\node (T') at (6,1) {$Y$};
\node (ST') at (4,1) {$X$};
\node (TS0') at (5,0) {$Z$};
\path[->] 
(ST') edge node[scale=.7] [left, xshift=-0.2cm, yshift=-0.1cm] {$F$} (TS0')
(T') edge node[scale=.7] [right, xshift=0.2cm, yshift=-0.1cm] (S) {$G$} (TS0');
\end{tikzpicture}
\end{center}
When it will be clear from context, we will sometimes avoid saying explicitly which spans/cospans we are considering and might refer to them as $[F,\,G]$. For 2-categorical background we redirect the reader to \cite{Gray:Adj, Lack:comp}.

Let us recall the definition of a relative monad \cite[Definition~2.1]{AltChap:mnd-no-end}. 
\begin{defn}
\label{defn:rel-mnd-cat}
A \emph{relative monad} on a functor $I:\catC_0\rightarrow\catC$ consists of:
\begin{itemize}
\item an object mapping $S\colon\textrm{Ob}(\catC_0)\to\textrm{Ob}(\catC)$;
\item for any $x,\,y\in\catC_0$, a map $(-)^\ast_S\colon\catC(Ix,\,Sy)\rightarrow\catC(Sx,\,Sy)$ (the \emph{extension});
\item for any $x\in\catC_0$, a map $s_x\colon Ix\rightarrow Sx$ (the \emph{unit});
\end{itemize}
satisfying the following axioms 
\begin{itemize}
\item the \emph{left unital law}, i.e.\ for any $k\colon Ix\to Sy$, $k=k^\ast\cdot s_x$;
\item the \emph{right unital law}, i.e.\ for any $x\in\catC_0$, $s_x^\ast=1_{Sx}$; 
\item the \emph{associativity law}, i.e.\ for any $k\colon Ix\rightarrow Sy$ and $l\colon Iy\rightarrow Sz$, $(l^\ast\cdot k)^\ast=l^\ast\cdot k^\ast$.
\end{itemize}
\end{defn}

Given a relative monad, it follows that $S$ is functorial and $s$ and $(-)^\ast_S$ natural (see \cite{AltChap:mnd-no-end}). A relative monad with $I=1_{\catC}$ is a no-iteration monad (also called extension system \cite[Definition~2.3]{MarmWood:ext-syst}). The notion of a no-iteration monad can be generalised to any 2-category $\catK$ thanks to the definition of \emph{pasting operators} \cite[Definition~2.1]{MarmWood:ext-syst}, which is a mapping of 2-cells as shown below
\begin{center}
\begin{tikzpicture}
\node (T) at (0,2) {$O$};
\node (T') at (2,2) {$Y$};
\node (TS0') at (2,0) {$Z$};
\path[->] 
(T) edge[dotted, bend right] node[scale=.7] [left, xshift=-0.5cm, yshift=0.5cm] (A) {$A$} (TS0')
	edge[dotted] node[scale=.7] [above] {$B$} (T')
(T') edge node[scale=.7] [right] (S) {$S$} (TS0');
\draw[-{Implies},double distance=1.5pt,shorten >=15pt,shorten <=15pt] (A) to node[scale=.7] [above] {$f$} (S);

\node (T) at (5,2) {$O$};
\node (T') at (7,2) {$Y$};
\node (ST') at (5,0) {$Z$};
\node (TS0') at (7,0) {$Z'$,};
\path[->] 
(ST') edge node[scale=.7] [below] {$T$} (TS0')
(T) edge[dotted] node[scale=.7] [left] (A') {$A$} (ST')
	edge[dotted] node[scale=.7] [above] {$B$} (T')
(T') edge node[scale=.7] [right] (U) {$U$} (TS0');
\draw[-{Implies},double distance=1.5pt,shorten >=20pt,shorten <=20pt] (A') to node[scale=.7] [above] {$f^\#$} (U);

\node (a) at (3,1) {};
\node (b) at (4,1) {};
\draw[->,decorate,decoration=snake] (a) to node[scale=.7] [above,yshift=0.2cm] {$(-)^\#$} (b);
\end{tikzpicture}
\end{center}
indexed on spans $(A,B)\colon O\rightarrow Z;Y$ and satisfying two axioms. Similarly, the extension of a relative monad can be expressed as a mapping of the form
\begin{center}
\begin{tikzpicture}
\node (T) at (0,2) {\textbf{1}};
\node (T') at (2,2) {$\catC_0$};
\node (ST') at (0,0) {$\catC_0$};
\node (TS0') at (2,0) {$\catC$};
\path[->] 
(ST') edge node[scale=.7] [below] {$I$} (TS0')
(T) edge[dotted] node[scale=.7] [left] (A) {$x$} (ST')
	edge[dotted] node[scale=.7] [above] {$y$} (T')
(T') edge node[scale=.7] [right] (S) {$S$} (TS0');
\draw[-{Implies},double distance=1.5pt,shorten >=20pt,shorten <=20pt] (A) to node[scale=.7] [above] {$k$} (S);

\node (T) at (5,2) {\textbf{1}};
\node (T') at (7,2) {$\catC_0$};
\node (ST') at (5,0) {$\catC_0$};
\node (TS0') at (7,0) {$\catC$,};
\path[->] 
(ST') edge node[scale=.7] [below] {$S$} (TS0')
(T) edge[dotted] node[scale=.7] [left] (A') {$x$} (ST')
	edge[dotted] node[scale=.7] [above] {$y$} (T')
(T') edge node[scale=.7] [right] (U) {$S$} (TS0');
\draw[-{Implies},double distance=1.5pt,shorten >=20pt,shorten <=20pt] (A') to node[scale=.7] [above] {$k^\ast_S$} (U);

\node (a) at (3,1) {};
\node (b) at (4,1) {};
\draw[->,decorate,decoration=snake] (a) to node[scale=.7] [above,yshift=0.2cm] {$(-)^\ast_S$} (b);
\end{tikzpicture}
\end{center}
where \textbf{1} is the terminal category and $x,\,y\colon\textbf{1}\to\catC_0$ are the constant functors to $x$ and $y$ respectively, satisfying two axioms. With this in mind, we can see how the next definition generalises pasting operators and gives us a way to define a relative monad in any 2-category. 
\begin{defn}
\label{defn:operator}
Let 
\begin{center}
\begin{tikzpicture}
\node (T) at (1,0) {$Z$};
\node (T') at (2,1) {$Y$};
\node (ST') at (0,1) {$X$};
\path[->] 
(ST') edge node[scale=.7] [left, xshift=-0.2cm, yshift=-0.1cm] (A) {$F$} (T)
(T')	edge node[scale=.7] [right, xshift=0.2cm, yshift=-0.1cm] {$G$} (T);

\node (T') at (6,1) {$Y$};
\node (ST') at (4,1) {$X$};
\node (TS0') at (5,0) {$Z'$};
\path[->] 
(ST') edge node[scale=.7] [left, xshift=-0.2cm, yshift=-0.1cm] {$F'$} (TS0')
(T') edge node[scale=.7] [right, xshift=0.2cm, yshift=-0.1cm] (S) {$G'$} (TS0');
\end{tikzpicture}
\end{center}
be two cospans in a 2-category $\catK$. An \emph{operator} $(-)^\#\colon [F,\,G]\rightarrow[F',\,G']$ is a family of functions, for any span of arrows $(A,\,B)\colon O\rightarrow X;Y$
$$(-)^\#_{A,B}:\catK[O,Z](FA,\,GB)\rightarrow\catK[O,Z'](F'A,\,G'B)$$
\begin{center}
\begin{tikzpicture}
\node (T) at (0,2) {$O$};
\node (T') at (2,2) {$Y$};
\node (ST') at (0,0) {$X$};
\node (TS0') at (2,0) {$Z$};
\path[->] 
(ST') edge node[scale=.7] [below] {$F$} (TS0')
(T) edge[dotted] node[scale=.7] [left] (A) {$A$} (ST')
	edge[dotted] node[scale=.7] [above] {$B$} (T')
(T') edge node[scale=.7] [right] (S) {$G$} (TS0');
\draw[-{Implies},double distance=1.5pt,shorten >=20pt,shorten <=20pt] (A) to node[scale=.7] [above] {$f$} (S);

\node (T) at (5,2) {$O$};
\node (T') at (7,2) {$Y$};
\node (ST') at (5,0) {$X$};
\node (TS0') at (7,0) {$Z'$,};
\path[->] 
(ST') edge node[scale=.7] [below] {$F'$} (TS0')
(T) edge[dotted] node[scale=.7] [left] (A') {$A$} (ST')
	edge[dotted] node[scale=.7] [above] {$B$} (T')
(T') edge node[scale=.7] [right] (U) {$G'$} (TS0');
\draw[-{Implies},double distance=1.5pt,shorten >=20pt,shorten <=20pt] (A') to node[scale=.7] [above] {$f^\#$} (U);

\node (a) at (3,1) {};
\node (b) at (4,1) {};
\draw[->,decorate,decoration=snake] (a) to node[scale=.7] [above,yshift=0.2cm] {$(-)^\#$} (b);
\end{tikzpicture}
\end{center}
satisfying the following axioms:
\begin{itemize}
\item \emph{indexing naturality}, i.e.\ for any diagram \raisebox{-40pt}{
\begin{tikzpicture}
\node (O') at (-1,2.5) {$O'$};
\node (T) at (0,2) {$O$};
\node (T') at (2,2) {$Y$};
\node (ST') at (0,0) {$X$};
\node (TS0') at (2,0) {$Z$};
\path[->] 
(ST') edge node[scale=.7] [below] {$F$} (TS0')
(T) edge[dotted] node[scale=.7] [left] (A) {$A$} (ST')
	edge[dotted] node[scale=.7] [above] {$B$} (T')
(T') edge node[scale=.7] [right] (S) {$G$} (TS0')
(O') edge node[scale=.7] [above, xshift=0.2cm] {$P$} (T);
\draw[-{Implies},double distance=1.5pt,shorten >=20pt,shorten <=20pt] (A) to node[scale=.7] [above] {$f$} (S);
\end{tikzpicture} }
, $(fP)^\#=f^\#P$;

\item \emph{left naturality}, i.e.\ for any diagram 
\raisebox{-35pt}{
\begin{tikzpicture}
\node (T) at (0,2) {$O$};
\node (T') at (2,2) {$Y$};
\node (ST') at (0,0) {$X$};
\node (TS0') at (2,0) {$Z$};
\path[->] 
(ST') edge node[scale=.7] [below] {$F$} (TS0')
(T) edge[dotted, bend left] node[scale=.7] [right] (A) {$A$} (ST')
	edge[dotted, bend right] node[scale=.7] [left] (A') {$A'$} (ST')
	edge[dotted] node[scale=.7] [above] {$B$} (T')
(T') edge node[scale=.7] [right] (S) {$G$} (TS0');
\draw[-{Implies},double distance=1.5pt,shorten >=8pt,shorten <=8pt] (A) to node[scale=.7] [above] {$f$} (S);
\draw[-{Implies},double distance=1.5pt,shorten >=5pt,shorten <=5pt] (A') to node[scale=.7] [above, yshift=0.1cm] {$\alpha$} (A);
\end{tikzpicture}  }
, $(f\cdot F\alpha)^\#=f^\#\cdot F'\alpha$; 
 
\item \emph{right naturality}, i.e.\ for any diagram  \raisebox{-35pt}{
\begin{tikzpicture}
\node (T) at (0,2) {$O$};
\node (T') at (2,2) {$Y$};
\node (ST') at (0,0) {$X$};
\node (TS0') at (2,0) {$Z$};
\path[->] 
(ST') edge node[scale=.7] [below] {$F$} (TS0')
(T) edge[dotted] node[scale=.7] [left] (A) {$A$} (ST')
	edge[dotted, bend right] node[scale=.7] [below] (B) {$B$} (T')
		edge[dotted, bend left] node[scale=.7] [above] (B') {$B'$} (T')
(T') edge node[scale=.7] [right] (S) {$G$} (TS0');
\draw[-{Implies},double distance=1.5pt,shorten >=20pt,shorten <=20pt] (A) to node[scale=.7] [below] {$f$} (S);
\draw[-{Implies},double distance=1.5pt,shorten >=5pt,shorten <=5pt] (B) to node[scale=.7] [right, xshift=0.1cm] {$\beta$} (B');
\end{tikzpicture} }
, $(G\beta\cdot f)^\#=G'\beta\cdot f^\#$. 
\end{itemize}
\end{defn}

The axioms of indexing, left and right naturality represent naturality in $O$, $A$ and~$B$ respectively. When we consider $F=1_X$ (so an operator $[1_X,\,G]\to[F',\,G']$) we get back the definition of \textit{pasting operator} given in \cite{MarmWood:ext-syst}. The conditions of whiskering and blistering of \cite{MarmWood:ext-syst} correspond to indexing and left naturality, while right naturality is deducible from \cite[Lemma~2.2]{MarmWood:ext-syst} and the interchange law of $\catK$. Pasting operators are also studied in \cite{HernER:No-it-distr}, where both left and right pasting operators are introduced. Following the reasoning above we can see that right pasting operators are equivalent to operators in $\catK^\op$ with $F=1_X$. 

\begin{eg} 
\label{eg:op} \hfill
\begin{enumerate}[(i)]
\item \label{eg:op-cat} Let us consider $\catK=\Cat$. We will show that, in this particular 2-category, an operator is equivalent to a family of maps indexed by pairs of objects. Let $\catX,\,\catY,\,\catZ$ and $\catZ'$ be categories and $(F,G)\colon \catY;\catX\rightarrow \catZ$ and $(F',G')\colon \catY;\catX\rightarrow \catZ'$ be two cospans in $\Cat$.
Let $(-)^\#\colon [F,\,G]\rightarrow[F',\,G']$ be an operator. Then, if we consider the span given by $(x,\,y)\colon\textbf{1}\to\catX;\catY$ with $x\in\catX$ and $y\in\catY$, the operator $(-)^\#$ gives us a family of maps
\begin{center}
$(-)_{x,y}^\#\colon \catZ(Fx,Gy)\rightarrow \catZ'(F'x,G'y)$ \\
\begin{tikzpicture}
\node (T) at (0,2) {\textbf{1}};
\node (T') at (2,2) {$\catY$};
\node (ST') at (0,0) {$\catX$};
\node (TS0') at (2,0) {$\catZ$};
\path[->] 
(ST') edge node[scale=.7] [below] {$F$} (TS0')
(T) edge[dotted] node[scale=.7] [left] (A) {$x$} (ST')
	edge[dotted] node[scale=.7] [above] {$y$} (T')
(T') edge node[scale=.7] [right] (S) {$G$} (TS0');
\draw[-{Implies},double distance=1.5pt,shorten >=20pt,shorten <=20pt] (A) to node[scale=.7] [above] {$f$} (S);

\node (T) at (5,2) {\textbf{1}};
\node (T') at (7,2) {$\catY$};
\node (ST') at (5,0) {$\catX$};
\node (TS0') at (7,0) {$\catZ'$};
\path[->] 
(ST') edge node[scale=.7] [below] {$F'$} (TS0')
(T) edge[dotted] node[scale=.7] [left] (A') {$x$} (ST')
	edge[dotted] node[scale=.7] [above] {$y$} (T')
(T') edge node[scale=.7] [right] (U) {$G'$} (TS0');
\draw[-{Implies},double distance=1.5pt,shorten >=20pt,shorten <=20pt] (A') to node[scale=.7] [above] {$f^\#$} (U);

\node (a) at (3,1) {};
\node (b) at (4,1) {};
\draw[->,decorate,decoration=snake] (a) to node[scale=.7] [above,yshift=0.2cm] {$(-)^\#$} (b);
\end{tikzpicture}
\end{center}
Left and right naturality of $(-)^\#$ tell us that these maps are natural in $x$ and $y$ respectively. Conversely, if we have such a natural family of maps, then we can construct a pasting operator in the following way. For any span $(A,\,B)\colon \mathbb{O}\rightarrow\catX;\catY$ and any natural transformation $f\colon FA\rightarrow GB$, we define the component of the natural transformation $f^\#\colon FA\rightarrow GB$ at $o\in\mathbb{O}$ as 
\begin{center}
$(f)^\#_{Ao,Bo}\colon FAo\longrightarrow GBo$ \\
\begin{tikzpicture}
\node (T) at (0,2) {\textbf{1}};
\node (T') at (2,2) {$\catY$};
\node (ST') at (0,0) {$\catX$};
\node (TS0') at (2,0) {$\catZ$};
\path[->] 
(ST') edge node[scale=.7] [below] {$F$} (TS0')
(T) edge[dotted] node[scale=.7] [left] (A) {$Ao$} (ST')
	edge[dotted] node[scale=.7] [above] {$Bo$} (T')
(T') edge node[scale=.7] [right] (S) {$G$} (TS0');
\draw[-{Implies},double distance=1.5pt,shorten >=20pt,shorten <=20pt] (A) to node[scale=.7] [above, yshift=0.1cm] {$f$} (S);

\node (T) at (5,2) {\textbf{1}};
\node (T') at (7,2) {$\catY$};
\node (ST') at (5,0) {$\catX$};
\node (TS0') at (7,0) {$\catZ'.$};
\path[->] 
(ST') edge node[scale=.7] [below] {$F'$} (TS0')
(T) edge[dotted] node[scale=.7] [left] (A') {$Ao$} (ST')
	edge[dotted] node[scale=.7] [above] {$Bo$} (T')
(T') edge node[scale=.7] [right] (U) {$G'$} (TS0');
\draw[-{Implies},double distance=1.5pt,shorten >=20pt,shorten <=20pt] (A') to node[scale=.7] [above, yshift=0.1cm] {$(f)^\#_{Ao,Bo}$} (U);

\node (a) at (3,1) {};
\node (b) at (4,1) {};
\draw[->,decorate,decoration=snake] (a) to node[scale=.7] [above,yshift=0.2cm] {$(-)^\#_{Ao,Bo}$} (b);
\end{tikzpicture}
\end{center}
Using naturality in $a$ and $b$ and naturality for $f$ we can prove that $f^\#$ is also a natural transformation. Moreover this definition satisfies all the axioms of an operator: indexing naturality follows directly from the definition, left and right naturality follow from naturality in $a$ and $b$ respectively.
\item \label{eg:enr-oper}
Let us look at the notion of operator when we set $\catK=\vcat$, the 2-category of $\mathcal{V}$-categories with $\mathcal{V}$ a monoidal category. For background in enriched category theory we redirect the reader to \cite{KellyG:bascec}.

We have a description similar to the one in the previous example. Let $\catX,\,\catY,\,\catZ$ and $\catZ'$ be $\mathcal{V}$-categories and $(F,G)\colon \catY;\catX\rightarrow \catZ$ and $(F',G')\colon \catY;\catX\rightarrow \catZ'$ be two cospans in $\vcat$. Also in this case an operator $(-)^\#\colon [F,\,G]\rightarrow[F',\,G']$ is equivalent to a family of functions, for any $a\in \catX$ and $b\in \catY$, 
$$(-)_{a,b}^\#\colon \overline{\catZ}(Fa,Gb)\longrightarrow \overline{\catZ'}(F'a,G'b)$$
natural in $a$ and $b$, where $\overline{\catZ}$ and $\overline{\catZ'}$ are the underlying categories of $\catZ$ and $\catZ'$, respectively. We get this characterisation by setting $\mathbb{O}=\mathbb{I}$ the unit $\mathcal{V}$-category, which has one object and the monoidal unit $I\in\mathcal{V}$ as hom-object. 

\item \label{eq:enr-fam-op} Let us consider another important example in $\catK=\vcat$ (we will use the same notation as above). If we have a natural family of maps in $\mathcal{V}$, for any $a\in \catX$ and $b\in \catY$, 
$$(-)_{a,b}^\#\colon\catZ(Fa,Gb)\longrightarrow \catZ'(F'a,G'b),$$
then we can construct a pasting operator in the following way. For any span $(A,\,B)\colon\mathbb{O}\rightarrow\catX;\catY$ and any $\mathcal{V}$-natural transformation $f\colon FA\rightarrow GB$, then we define the $o\in\mathbb{O}$ component of $f^\#$ as 
\begin{center}
$I\xrightarrow{f_o}\catZ(FAo,\,GBo)\xrightarrow{(-)^\#_{Ao,Bo}}\catZ'(F'Ao,\,G'Bo)$.
\end{center}
Using naturality in $a$ and $b$ and $\mathcal{V}$-naturality for $f$ we can prove that $f^\#$ is also a $\mathcal{V}$-natural transformation. Moreover this definition satisfies all the axioms of an operator: indexing naturality follows directly from the definition, left and right naturality follow from naturality in $a$ and $b$ respectively.

\item \yellow Let $\mathcal{V}$ be a strict monoidal category and let us consider $\catK=\Sigma\mathcal{V}$ the one-object 2-category with hom-category $\mathcal{V}$. Then, two cospans in $\Sigma\mathcal{V}$ 
\begin{center}
\begin{tikzpicture}
\node (T) at (1,0) {$\ast$};
\node (T') at (2,1) {$\ast$};
\node (ST') at (0,1) {$\ast$};
\path[->] 
(ST') edge node[scale=.7] [left, xshift=-0.2cm, yshift=-0.1cm] (A) {$X$} (T)
(T')	edge node[scale=.7] [right, xshift=0.2cm, yshift=-0.1cm] {$Y$} (T);

\node (T') at (6,1) {$\ast$};
\node (ST') at (4,1) {$\ast$};
\node (TS0') at (5,0) {$\ast$};
\path[->] 
(ST') edge node[scale=.7] [left, xshift=-0.2cm, yshift=-0.1cm] {$X'$} (TS0')
(T') edge node[scale=.7] [right, xshift=0.2cm, yshift=-0.1cm] (S) {$Y'$} (TS0');
\end{tikzpicture}
\end{center}
are two pair of objects $X,Y$ and $X',Y'$ in $\mathcal{V}$. Thus, it is easy to check that an operator $[X,Y]\to[X',Y']$ in $\Sigma\mathcal{V}$ is a natural family of maps 
$$\mathcal{V}(X\otimes A,Y\otimes B)\to\mathcal{V}(X'\otimes A,Y'\otimes B),$$ i.e. a natural transformation between the profunctors 
$$\mathcal{V}(X\otimes -,Y\otimes -)\to\mathcal{V}(X'\otimes -,Y'\otimes -).$$ \black

\item \yellow Let $\catC$ be a category and let us consider $\catK=\tilde{\catC}$ the locally discrete 2-category associated to $\catC$. Then, a pair of maps $F,G\colon X\to Y$ in $\catC$ are jointly monic if and only if there exists an operator $[F,G]\to[1_X,1_X]$ in $\tilde{\catC}$. \black
\end{enumerate}
\end{eg}

Using operators we can also define relative adjunctions in $\catK$ as follows.  

\begin{defn}
Let $I\colon C_0\rightarrow C$ be a 1-cell in $\catK$. A \emph{relative adjunction in $\catK$} over $I$, denoted as $F {\,}_I\!\dashv G$ consists of  an object $D$ in $\catK$ together with:
\begin{itemize}
\item two 1-cells $F\colon C_0\rightarrow D$ and $G\colon D\rightarrow C$;
\item a 2-cell $\iota\colon I\rightarrow GF$;
\end{itemize}
such that the operator $G(-)\iota\colon [F,\,1_D]\rightarrow[I,G]$ induces isomorphisms, for any span $(A,B)\colon O\rightarrow C_0;D$,
$$\catK[O,\,D](FA,\,B)\xrightarrow{\sim} \catK[O,\,C](IA,\,GB)$$
\end{defn}

\begin{rmk}
By taking $I=1_C$ we see that this definition generalises that of an adjunction in $\catK$ \cite[Section~2.1]{Lack:comp}. Additionally, considering the case $\catK=$~$\Cat$ we see that it also generalises the definition of relative adjunction defined in \cite[Definition~2.2]{Ulm:prop-rel-adj}.
\end{rmk}

In \cite[Proposition~7]{StrWal:yon-str} we find an alternative definition of a relative adjunction using absolute left liftings. The next proposition shows that our definition with operators is equivalent to the one with liftings.

\begin{prop}
\label{prop:rel-adj-abs-lift}
Given the following diagram in $\catK$ \yellow 
\[\begin{tikzcd}[ampersand replacement=\&]
	\& D \\
	{C_0} \&\& C
	\arrow["F", from=2-1, to=1-2]
	\arrow["G", from=1-2, to=2-3]
	\arrow[""{name=0, anchor=center, inner sep=0}, "I"', from=2-1, to=2-3]
	\arrow["\iota"'{xshift=0.1cm}, shorten <=6pt, shorten >=6pt, Rightarrow, from=0, to=1-2]
\end{tikzcd}\] \black 
then, $F$ is an absolute left lifting of $I$ along $G$ if and only if $F_I\dashv G$. 
\end{prop}

\begin{proof}
Let us start assuming that $F$ is an absolute left lifting of $I$ along $G$ with universal 2-cell $\iota\colon I\to GF$. We need to show that the operator $G(-)\iota\colon[F,1]\to[I,G]$ is invertible, i.e. for any span $(A,B)\colon O\rightarrow C_0;D$,
$$\catK[O,\,D](FA,\,B)\xrightarrow{\sim} \catK[O,\,C](IA,\,GB).$$
First, let us notice that, since $F$ is an absolute left lifting, $(FA,\iota A)$ is a left lifting of $IA$ along $G$. Therefore, for any 2-cell $\beta\colon IA\to GB$ be a 2-cell, there exists a unique $\alpha\colon FA\to B$ such that $G\alpha\cdot\iota A=\beta$. Hence, $G(-)\iota$ is invertible. 

On the other hand, if we have an invertible operator $G(-)\iota\colon[F,1]\to[I,G]$ we can prove that $\iota$ provides $F$ as an absolute left lifting of $I$ along $G$.
\begin{itemize}
\item First, let us prove that $F$ is a left lifting. For any 2-cell $\beta$ of the form
\begin{center}
\begin{tikzcd}[ampersand replacement=\&]
	\& D \\
	{C_0} \& C
	\arrow[""{name=0, anchor=center, inner sep=0}, "G", from=1-2, to=2-2]
	\arrow[""{name=1, anchor=center, inner sep=0}, "I"', from=2-1, to=2-2]
	\arrow["B", dashed, from=2-1, to=1-2]
	\arrow["\beta", shorten <=7pt, shorten >=7pt, Rightarrow, from=1, to=0]
\end{tikzcd}
\hspace{0.5cm}$=$\hspace{0.5cm}
\begin{tikzcd}[ampersand replacement=\&]
	{C_0} \& D \\
	{C_0} \& C,
	\arrow[""{name=0, anchor=center, inner sep=0}, "{1_{C_0}}"', Rightarrow, dashed, no head, from=1-1, to=2-1]
	\arrow["B", dashed, from=1-1, to=1-2]
	\arrow[""{name=1, anchor=center, inner sep=0}, "G", from=1-2, to=2-2]
	\arrow["I"', from=2-1, to=2-2]
	\arrow["\beta", shorten <=13pt, shorten >=13pt, Rightarrow, from=0, to=1]
\end{tikzcd}
\end{center}
since $[G(-)\iota]_{1_{C_0},B}$ is invertible, there exists a unique $\alpha\colon F1_{C_0}=F\to B$ such that $G(\alpha)\iota=\beta$, which is the \yellow required \black lifting property. 

\item Second, we need to prove that this lifting is absolute, i.e. that for any 1-cell $A\colon O\to C_0$, then $FA$ is a left lifting of $IA$ along $G$. This again, follows by the isomorphism
\[G(-)\iota\colon\catK[O,\,D](FA,\,B)\xrightarrow{\sim} \catK[O,\,C](IA,\,GB).\hfill\qedhere\]
\end{itemize}
\end{proof}
\black 

Let us underline the fact that we have not used the naturality axioms for an operator in the proof of the proposition above. This is because we were considering $G(-)\iota$ which is always an operator. 

\begin{rmk}
\label{rmk:weber-fully-faith}
Weber in \cite[Example~2.18]{web:yon-str} defines a fully faithful 1-cell $F\colon X\to Y$ in a 2-category $\catK$ as a 1-cell such that the identity cell
\[\begin{tikzcd}
	X & X \\
	Y
	\arrow["{1_X}", from=1-1, to=1-2]
	\arrow[""{name=0, anchor=center, inner sep=0}, "F"', from=1-1, to=2-1]
	\arrow[""{name=1, anchor=center, inner sep=0}, "F", from=1-2, to=2-1]
	\arrow[shorten <=5pt, shorten >=5pt, Rightarrow, no head, from=0, to=1]
\end{tikzcd}\]
exhibits $1_X$ as an absolute left lifting of $F$ along itself. Using \Cref{prop:rel-adj-abs-lift}, this is equivalent to $(1_X)_F\dashv F$, i.e. \yellow $F$ induces \black an invertible operator $[1_X,1_X]\cong[F,F]$. 
\end{rmk}

The next proposition gives a characterisation of operators in 2-categories with comma objects. For the comma object of $F$ and $G$ we will use the following notation:
\[\begin{tikzcd}
	{F/G} & Y \\
	X & Z.
	\arrow["{p_Y}", from=1-1, to=1-2]
	\arrow[""{name=0, anchor=center, inner sep=0}, "{p_X}"', from=1-1, to=2-1]
	\arrow["F"', from=2-1, to=2-2]
	\arrow[""{name=1, anchor=center, inner sep=0}, "G", from=1-2, to=2-2]
	\arrow["\rho", shorten <=13pt, shorten >=13pt, Rightarrow, from=0, to=1]
\end{tikzcd}\]

\begin{prop}
\label{prop:op-w-comma}
Let $\catK$ be a 2-category. Given two cospan 
\begin{center}
\begin{tikzpicture}
\node (T) at (1,0) {$Z$};
\node (T') at (2,1) {$Y$};
\node (ST') at (0,1) {$X$};
\path[->] 
(ST') edge node[scale=.7] [left, xshift=-0.2cm, yshift=-0.1cm] (A) {$F$} (T)
(T')	edge node[scale=.7] [right, xshift=0.2cm, yshift=-0.1cm] {$G$} (T);

\node (T') at (6,1) {$Y$};
\node (ST') at (4,1) {$X$};
\node (TS0') at (5,0) {$Z'$};
\path[->] 
(ST') edge node[scale=.7] [left, xshift=-0.2cm, yshift=-0.1cm] {$F'$} (TS0')
(T') edge node[scale=.7] [right, xshift=0.2cm, yshift=-0.1cm] (S) {$G'$} (TS0');
\end{tikzpicture}
\end{center}
if the comma objects $F/G$ and $F'/G'$ exist, then operators $(-)^\#\colon[F,G]\to[F',G']$ are equivalent to 1-cells $H\colon F/G\to F'/G'$ such that 
\begin{equation}
\label{eq:iso-op-w-comma}
\begin{tikzcd}[ampersand replacement=\&]
	{F/G} \& {F'/G'} \& {\textrm{and}} \& {F/G} \\
	\& X \&\& {F'/G'} \& Y.
	\arrow["{p'_X}", from=1-2, to=2-2]
	\arrow["H", from=1-1, to=1-2]
	\arrow[""{name=0, anchor=center, inner sep=0}, "{p_X}"', curve={height=12pt}, from=1-1, to=2-2]
	\arrow["H"', from=1-4, to=2-4]
	\arrow["{p'_Y}"', from=2-4, to=2-5]
	\arrow[""{name=1, anchor=center, inner sep=0}, "{p_Y}", curve={height=-12pt}, from=1-4, to=2-5]
	\arrow["\cong"{description}, Rightarrow, draw=none, from=0, to=1-2]
	\arrow["\cong"{description}, Rightarrow, draw=none, from=1, to=2-4]
\end{tikzcd}
\end{equation}
\end{prop}

\begin{proof}
Let us start with an operator $(-)^\#\colon[F,G]\to[F',G']$. Setting $O:=F/G$, $A:=p_X$ and $B:=p_Y$, we can apply the operator $(-)^\#$ to the 2-cell $\rho$.
\begin{center}
\begin{tikzpicture}
\node (T) at (0,2) {$F/G$};
\node (T') at (2,2) {$Y$};
\node (ST') at (0,0) {$X$};
\node (TS0') at (2,0) {$Z$};
\path[->] 
(ST') edge node[scale=.7] [below] {$F$} (TS0')
(T) edge[dotted] node[scale=.7] [left] (A) {$p_X$} (ST')
	edge[dotted] node[scale=.7] [above] {$p_Y$} (T')
(T') edge node[scale=.7] [right] (S) {$G$} (TS0');
\draw[-{Implies},double distance=1.5pt,shorten >=20pt,shorten <=20pt] (A) to node[scale=.7] [above, yshift=0.1cm] {$\rho$} (S);

\node (T) at (5,2) {$F/G$};
\node (T') at (7,2) {$Y$};
\node (ST') at (5,0) {$X$};
\node (TS0') at (7,0) {$Z'$};
\path[->] 
(ST') edge node[scale=.7] [below] {$F'$} (TS0')
(T) edge[dotted] node[scale=.7] [left] (A') {$p_X$} (ST')
	edge[dotted] node[scale=.7] [above] {$p_Y$} (T')
(T') edge node[scale=.7] [right] (U) {$G'$} (TS0');
\draw[-{Implies},double distance=1.5pt,shorten >=20pt,shorten <=20pt] (A') to node[scale=.7] [above, yshift=0.1cm] {$\rho^\#$} (U);

\node (a) at (3,1) {};
\node (b) at (4,1) {};
\draw[->,decorate,decoration=snake] (a) to node[scale=.7] [above,yshift=0.2cm] {$(-)^\#_{p_X,p_Y}$} (b);
\end{tikzpicture}
\end{center}
Then, by the universal property of the comma object $F'/G'$, there exists a unique $H^\#\colon F/G\to F'/G'$ such that
\[\begin{tikzcd}
	{F/G} &&&& {F/G} \\
	&& Y & {=} && {F'/G'} & Y \\
	& X & {Z'} &&& X & {Z'}
	\arrow["{G'}", from=2-3, to=3-3]
	\arrow["{F'}"', from=3-2, to=3-3]
	\arrow["{p_Y}", curve={height=-12pt}, from=1-1, to=2-3]
	\arrow[""{name=0, anchor=center, inner sep=0}, "{p_X}"', curve={height=12pt}, from=1-1, to=3-2]
	\arrow[""{name=1, anchor=center, inner sep=0}, "{G'}", from=2-7, to=3-7]
	\arrow["{F'}"', from=3-6, to=3-7]
	\arrow["{p'_Y}", from=2-6, to=2-7]
	\arrow[""{name=2, anchor=center, inner sep=0}, "{p'_X}"', from=2-6, to=3-6]
	\arrow[""{name=3, anchor=center, inner sep=0}, "{p_Y}", curve={height=-12pt}, from=1-5, to=2-7]
	\arrow[""{name=4, anchor=center, inner sep=0}, "{p_X}"', curve={height=12pt}, from=1-5, to=3-6]
	\arrow["{\exists!H^\#}"{description}, from=1-5, to=2-6]
	\arrow["{\rho^\#}", shorten <=20pt, shorten >=20pt, Rightarrow, from=0, to=2-3]
	\arrow["\cong"{description}, Rightarrow, draw=none, from=4, to=2-6]
	\arrow["\cong"{description}, Rightarrow, draw=none, from=2-6, to=3]
	\arrow["{\rho'}", shorten <=14pt, shorten >=14pt, Rightarrow, from=2, to=1]
\end{tikzcd}\]
Conversely, we now show how to construct an operator $(-)^H\colon [F,G]\to[F',G']$ given a 1-cell $H\colon F/G\to F'/G'$ with isomorphisms as in (\ref{eq:iso-op-w-comma}). Let $(A,B)\colon O\to X;Y$ be a span and $f\colon FA\Rightarrow GB$ a 2-cell in $\catK$. By the universal property of the comma object $F/G$, there exists a unique $T^f\colon O\to F/G$ such that
\[\begin{tikzcd}
	O &&&& O \\
	&& Y & {=} && {F/G} & Y \\
	& X & Z &&& X & Z.
	\arrow["G", from=2-3, to=3-3]
	\arrow["F"', from=3-2, to=3-3]
	\arrow["B", curve={height=-12pt}, from=1-1, to=2-3]
	\arrow[""{name=0, anchor=center, inner sep=0}, "A"', curve={height=12pt}, from=1-1, to=3-2]
	\arrow[""{name=1, anchor=center, inner sep=0}, "G", from=2-7, to=3-7]
	\arrow["F"', from=3-6, to=3-7]
	\arrow["{p_Y}", from=2-6, to=2-7]
	\arrow[""{name=2, anchor=center, inner sep=0}, "{p_X}"', from=2-6, to=3-6]
	\arrow[""{name=3, anchor=center, inner sep=0}, "B", curve={height=-12pt}, from=1-5, to=2-7]
	\arrow[""{name=4, anchor=center, inner sep=0}, "A"', curve={height=12pt}, from=1-5, to=3-6]
	\arrow["{\exists!T^f}"{description}, from=1-5, to=2-6]
	\arrow["f", shorten <=20pt, shorten >=20pt, Rightarrow, from=0, to=2-3]
	\arrow["\cong"{description}, Rightarrow, draw=none, from=4, to=2-6]
	\arrow["\cong"{description}, Rightarrow, draw=none, from=2-6, to=3]
	\arrow["\rho", shorten <=13pt, shorten >=13pt, Rightarrow, from=2, to=1]
\end{tikzcd}\]
Then, we define the 2-cell $f^H$ as the following pasting
\[\begin{tikzcd}
	O \\
	& {F/G} \\
	&& {F'/G'} & Y \\
	&& X & {Z'.}
	\arrow[""{name=0, anchor=center, inner sep=0}, "{G'}", from=3-4, to=4-4]
	\arrow["{F'}"', from=4-3, to=4-4]
	\arrow["{p'_Y}", from=3-3, to=3-4]
	\arrow[""{name=1, anchor=center, inner sep=0}, "{p'_X}"', from=3-3, to=4-3]
	\arrow[""{name=2, anchor=center, inner sep=0}, "{p_Y}"{description}, curve={height=-12pt}, from=2-2, to=3-4]
	\arrow[""{name=3, anchor=center, inner sep=0}, "{p_X}"{description}, curve={height=12pt}, from=2-2, to=4-3]
	\arrow["H"{description}, from=2-2, to=3-3]
	\arrow["{T^f}"{description}, from=1-1, to=2-2]
	\arrow[""{name=4, anchor=center, inner sep=0}, "B", curve={height=-24pt}, from=1-1, to=3-4]
	\arrow[""{name=5, anchor=center, inner sep=0}, "A"', curve={height=24pt}, from=1-1, to=4-3]
	\arrow["\cong"{description}, Rightarrow, draw=none, from=3, to=3-3]
	\arrow["\cong"{description}, Rightarrow, draw=none, from=3-3, to=2]
	\arrow["{\rho'}", shorten <=14pt, shorten >=14pt, Rightarrow, from=1, to=0]
	\arrow["\cong"{description}, Rightarrow, draw=none, from=2-2, to=4]
	\arrow["\cong"{description}, Rightarrow, draw=none, from=5, to=2-2]
\end{tikzcd}\]
Now we need to prove that $(-)^H$ satisfy all of the three naturality axioms for an operator.
\begin{itemize}
\item \textbf{Indexing naturality:} Given a 1-cell $P\colon O'\to O$, we need to prove that $f^HP=(fP)^H$. This is true because both these maps correspond to the 1-cell $T^fQ$ through the (1-dimensional) universal property of the comma object $F'/G'$. 

\item \textbf{Left naturality:} Given $\alpha\colon A'\to A$, we need to prove that $(f\cdot F\alpha)^H=f^H\cdot F'\alpha$. Let us denote with $T^f$ and $T^{f\cdot F\alpha}$ the 1-cells $O\to F/G$ corresponding to $f$ and $f\cdot F\alpha$, respectively. Then, one can check that we can construct a 2-cell $\overline{\alpha}\colon T^{f\cdot F\alpha}\to T^f$ corresponding to 
\[\begin{tikzcd}
	O &&&&& O & {F/G} \\
	\\
	{F/G} && {F/G} & Y & {=} && {F/G} & Y \\
	&& X & Z &&& X & Z.
	\arrow[""{name=0, anchor=center, inner sep=0}, "G", from=3-4, to=4-4]
	\arrow["F"', from=4-3, to=4-4]
	\arrow[""{name=1, anchor=center, inner sep=0}, "G", from=3-8, to=4-8]
	\arrow["F"', from=4-7, to=4-8]
	\arrow["{p_Y}"{description}, from=3-7, to=3-8]
	\arrow[""{name=2, anchor=center, inner sep=0}, "{p_X}"', from=3-7, to=4-7]
	\arrow[""{name=3, anchor=center, inner sep=0}, "B"{description}, from=1-6, to=3-8]
	\arrow[""{name=4, anchor=center, inner sep=0}, "{T^{f\cdot F\alpha}}"', from=1-6, to=3-7]
	\arrow["{T^f}", from=1-6, to=1-7]
	\arrow[""{name=5, anchor=center, inner sep=0}, "{p_Y}", from=1-7, to=3-8]
	\arrow[""{name=6, anchor=center, inner sep=0}, "{T^f}", curve={height=-12pt}, from=1-1, to=3-3]
	\arrow["{p_Y}", from=3-3, to=3-4]
	\arrow[""{name=7, anchor=center, inner sep=0}, "{p_X}"{description}, from=3-3, to=4-3]
	\arrow["{T^{f\cdot F\alpha}}"', from=1-1, to=3-1]
	\arrow["{p_X}"', curve={height=6pt}, from=3-1, to=4-3]
	\arrow[""{name=8, anchor=center, inner sep=0}, "{A'}"{description}, curve={height=18pt}, from=1-1, to=4-3]
	\arrow[""{name=9, anchor=center, inner sep=0}, "A"{description}, curve={height=-12pt}, from=1-1, to=4-3]
	\arrow["\rho", shorten <=13pt, shorten >=13pt, Rightarrow, from=2, to=1]
	\arrow["\cong"{description}, Rightarrow, draw=none, from=4, to=3]
	\arrow["\cong"{description}, Rightarrow, draw=none, from=3, to=5]
	\arrow["\rho", shorten <=13pt, shorten >=13pt, Rightarrow, from=7, to=0]
	\arrow["\cong"', Rightarrow, draw=none, from=3-1, to=8]
	\arrow["\cong"{description}, Rightarrow, draw=none, from=9, to=6]
	\arrow["\alpha", shorten <=8pt, shorten >=8pt, Rightarrow, from=8, to=9]
\end{tikzcd}\]
Therefore, $H\overline{\alpha}\colon HT^{f\cdot F\alpha}\to HT^f$ is a 2-cell between 1-cell with codomain $F'/G'$. Finally, using the isomorphism $p_X\overline{\alpha}\cong p'_XH\overline{\alpha}$, the 2-dimensional universal property of $F'/G'$ shows that the 2-cells $(f\cdot F\alpha)^H$ and $f^H\cdot F'\alpha$ are equal.

\item \textbf{Right naturality:} This is analogous to left naturality. \qedhere
\end{itemize}
\end{proof}

In $\Cat$, this proposition shows how families of maps $\catZ(Fx,Gy)\rightarrow \catZ'(F'x,G'y)$ natural in $x$ and $y$, correspond to functors $(F\downarrow G)\to(F'\downarrow G')$ between the comma categories. Let $\Int(\catE)$ be the 2-category of internal categories in a category with pullbacks $\catE$. It is known that $\Int(\catE)$ has comma objects, for example it follows from \cite[Proposition~3.19]{Bourke:phd} using the construction of comma objects via cotensors with $\mathbf{2}$ and pullbacks given in \cite{Street:fib-yon}. Some important examples of internal categories are double categories and (small) strict monoidal categories, where we take $\catE$ equal to $\Cat$ and to $\Mon$, the category of monoids, respectively. Hence, \Cref{prop:op-w-comma} gives a recipe to interpret operators also in these 2-categories. 

%
%
%
%

In \cite[Section B.I.2]{Law:phd} Lawvere gives a definition of adjoints using comma categories, which is equivalent to the classic one (see \cite[Theorem B.I.2.1]{Law:phd}). Putting together Propositions \ref{prop:rel-adj-abs-lift} and \ref{prop:op-w-comma} we get a similar description for relative adjunctions in our setting. 

\begin{prop}
Let $\catK$ be a 2-category with comma objects and 
\[\begin{tikzcd}
	& D \\
	{C_0} && C
	\arrow["L", from=2-1, to=1-2]
	\arrow["R", from=1-2, to=2-3]
	\arrow["J"', from=2-1, to=2-3]
\end{tikzcd}\]
a diagram in $\catK$. Then, $L_J\dashv R$ if and only if there exists an invertible 1-cell $L/1_D\xrightarrow{\sim}$~$J/R$ satisfying the equations (\ref{eq:iso-op-w-comma}). 
\end{prop}

\begin{proof}
Set $F=L$, $G=1$, $F'=I$ and $G'=R$ in \Cref{prop:op-w-comma}. 
\end{proof}

\begin{rmk} \label{rmk:pst-op}
Let us state three properties that will be useful to prove that any relative adjunction induces a relative monad (Lemma~\ref{lemma:rel-adj-to-rel-mnd}). 
\begin{itemize}
\item We can easily see that given two operators $[F_1,\,G_1]\rightarrow[F_2,\,G_2]$ and $[F_2,\,G_2]\rightarrow[F_3,\,G_3]$ we can construct a composition operator $[F_1,\,G_1]\rightarrow[F_3,\,G_3]$ composing component-wise.

\item An example of an operator is, for any 1-cell $T\colon Z\rightarrow Z'$ in $\catK$, $T(-)\colon [F,\,G]\rightarrow[TF,\,TG]$ defined as $Tf\colon TFA\to TGB$ for any $f\colon FA\to GB$. Indexing, left and right naturality in this case derive from the unique interpretation of a pasting diagram in a 2-category. 

\item Let $(-)^\#\colon[F,\,G]\to[F',\,G']$ be an operator between the cospans $(F,G)\colon X;Y\rightarrow Z$ and $(F',G')\colon X;Y\rightarrow Z'$. Then, given any 1-cells $F_0\colon X_0\to X$ and $G_0\colon Y_0\to Y$, we can construct a new operator 
$$(-)^\#\cdot(F_0,G_0)\colon[FF_0,\,GG_0]\to[F'F_0,\,G'G_0].$$
For any span $(A,B)\colon O\to X_0;Y_0$ we define the action of $(-)^\#_{F_0,G_0}$ on a 2-cell $f\colon FF_0A\to GG_0B$ as
$$[(f)^\#\cdot(F_0,G_0)]_{A,B}:=(f)^\#_{F_0A,G_0B}.$$
All three naturality axioms for $(-)^\#\cdot(F_0,G_0)$ hold since they are particular cases of the ones of $(-)^\#$. 
\end{itemize}  
\end{rmk}

We will use the following proposition in Section~\ref{sec:rel-in-K} to prove that any relative adjunction induces a relative monad (Lemma~\ref{lemma:rel-adj-to-rel-mnd}). 

\begin{prop}
\label{prop:pst-op}
Let $(-)^\#\colon [F,\,G]\rightarrow[F',\,G']$ be an operator such that each $(-)^\#_{A,B}$ is an isomorphism. Then the family of functions sending any $g\in\catK[O,Z'](F'A,\,G'B)$ to the unique $f$ such that $f^\#_{A,B}=g$ forms an operator $(-)^\flat\colon [F',\,G']\rightarrow[F,\,G]$. 
\end{prop}

\begin{proof}
Let us check all the axioms for $(-)^\flat$. For any $P\colon O'\rightarrow O$,
\begin{center}
$f^\flat P=(fP)^\flat\quad\Longleftrightarrow\quad (f^\flat P)^\#=((fP)^\flat)^\#$.
\end{center}
The second equation is true using the fact that $(-)^\flat$ is locally the inverse of $(-)^\#$ and indexing naturality for $(-)^\#$. With the same reasoning we can prove that $(-)^\flat$ satisfies also the other axioms. 
\end{proof}

\section{Relative Monads in $\catK$}
\label{sec:rel-in-K}

Using Definition\til\ref{defn:operator} we can define a relative monad in any 2-category $\catK$ as follows. 

\begin{defn}
\label{defn:rel-mnd}
A \emph{relative monad} $(X,\,I,\,S,\,(-)^\dagger_S,\,s)$ in $\catK$ consists of a pair of objects $X,\,X_0\in\catK$ together with:
\begin{itemize}
\item two 1-cells $I, S\colon X_0\rightarrow X$ (we say that \emph{$S$ is a relative monad on $I$});
\item an operator $(-)^\dagger_S\colon [I,\,S]\rightarrow[S,\,S]$ (the \emph{extension operator});
\item a 2-cell $s\colon I\rightarrow S$ (the \emph{unit});
\end{itemize}
satisfying the following axioms: 
\begin{itemize}
\item the \emph{left unit law}, i.e.\ for any $A,\,B\colon O\rightarrow X_0$ \\\begin{tikzpicture}
\node (IX) at (0,2) {$IA$};
\node (SX) at (2,2) {$SA$};
\node (SY) at (2,0) {$SB$;};
\draw[->] (IX) to node[scale=.7] [above]{$sA$} (SX);
\draw[->] (SX) to node[scale=.7] [right]{$k^\dagger$} (SY);
\draw[->] (IX) to [bend right] node[scale=.7] (G'f) [left] {$k$} (SY);
\end{tikzpicture}
\item the \emph{right unit law}, i.e.\ $s^\dagger=1_S$;
\item the \emph{associativity law}, i.e.\ for any 2-cells $k\colon IA\rightarrow SB$ and $l\colon IB\rightarrow SC$ \\ \begin{tikzpicture}
\node (IX) at (0,2) {$SA$};
\node (SX) at (2,2) {$SB$};
\node (SY) at (2,0) {$SC$.};
\draw[->] (IX) to node[scale=.7] [above]{$k^\dagger$} (SX);
\draw[->] (SX) to node[scale=.7] [right]{$l^\dagger$} (SY);
\draw[->] (IX) to [bend right] node[scale=.7] (G'f) [left, xshift=-0.2cm, yshift=-0.1cm] {$(l^\dagger\cdot k)^\dagger$} (SY);
\end{tikzpicture}
\end{itemize}
\end{defn}
With an abuse of notation, we will refer to a relative monad $(X,\,I,\,S,\,(-)^\dagger_S,\,s)$ in $\catK$ only with $(X,\,I,\,S)$.
\begin{eg}
\label{eg:rel-mnd} \hfill
\begin{enumerate}[(i)]
\item \label{rmk:0-cells} Let us consider relative monads in $\catK$ with $X_0=X$ and $I=1_X$.  Since operators with $I=1_X$ are pasting operators, we get back exactly no-iteration monads in~$\catK$ \cite[Theorem~2.4]{MarmWood:ext-syst}. 

\item Thanks to part (\ref{eg:op-cat}) of Example~\ref{eg:op}, setting $\catK=\Cat$ gives back exactly the definition of relative monad given in \cite{AltChap:mnd-no-end} and recalled in Definition\til\ref{defn:rel-mnd-cat} here.

\item Using part (\ref{eg:enr-oper}) of Example~\ref{eg:op} we can write more explicitly what is a relative monad $(\catX,\,I,\,S)$ in $\vcat$. Such an object consists of a pair of $\mathcal{V}$-categories $\catX$ and $\catX_0$ together with:
\begin{itemize}
\item two $\mathcal{V}$-functors $I, S\colon \catX_0\rightarrow\catX$;
\item a family of functions for any $a,b\in\catX_0$, $(-)^\dagger_S\colon \overline{\catX}(Ia,\,Sb)\rightarrow \overline{\catX}(Sa,\,Sb)$ natural in $a$ and $b$ (with $\overline{\catX}$ the underlying category of $\catX$);
\item a $\mathcal{V}$-natural transformation $s\colon I\rightarrow S$;
\end{itemize}
satisfying the left/right unital laws and associativity. 

We notice that this is not an \emph{enriched relative monad} \cite[Definition~4]{Stat:alg-pres}, which involves a natural family of maps in $\mathcal{V}$, for any $a,b\in \catX_0$, 
$$(-)_{a,b}^\#\colon\catX(Ia,Sb)\longrightarrow \catX(Sa,Sb).$$
Using the operators described in part (\ref{eq:enr-fam-op}) of \Cref{eg:op}, we can see that an enriched relative monads gives rise to a relative monad in $\vcat$. It would be interesting to investigate when the opposite is possible, i.e. when we can construct an enriched relative monad starting from a relative monad in $\vcat$. For instance, if the monoidal unit $I$ is a dense generator then $\mathcal{V}(I,-)$ is fully faithful and so natural families of maps $\catX(Ia,Sb)\to\catX(Sa,Sb)$ in $\mathcal{V}$ correspond to natural families of maps $\overline{\catX}(Ia,Sb)\to\overline{\catX}(Sa,Sb)$.

\item We might call relative monads in $\catK=\Int(\catE)$ \emph{internal relative monads}. In particular, for $\catE=\Cat$ we get \emph{double relative monads} and for $\catE=\Mon$ \emph{monoidal relative monads}.  

\item \yellow Let $\catK=\Pos$ be the 2-category with objects posets, 1-cells order-preserving maps and where there exists a 2-cell between two 1-cells $f$ and $g$ if and only if $f\leq g$. Then, a relative monad in $\Pos$ consists of two posets $X_0$ and $X$ together with
\begin{itemize}
\item an order preserving map $I\colon X_0\to X$;
\item for any $a\in X_0$, an element $Ta\in X$;
\end{itemize}
such that
$$Ta=\textrm{Min}\lbrace x\in X\mid \,\exists\,b\in X_0\,\textrm{such that}\,Tb=x\,\textrm{and}\,Ia\leq x\rbrace.$$ \black

\item \yellow Let $\catC$ be a category and let us consider $\catK=\tilde{\catC}$ the locally discrete 2-category associated to $\catC$. We recall that a monad in $\tilde{\catC}$ is just an object in $\catC$. Instead, a relative monad in $\tilde{\catC}$ is a morphism in $\catC$. \black


\end{enumerate}
\end{eg}

When we set $\catK=\Cat$ we know that any relative monad is induced by a relative adjunction \cite{AltChap:mnd-no-end}. It is natural to wonder if the same holds in any 2-category~$\catK$.  

\begin{lemma}
\label{lemma:rel-adj-to-rel-mnd}
A relative adjunction $F {\,}_I\!\dashv G$
induces a relative monad $(X,\,I,\,GF)$. 
\end{lemma}

\begin{proof}
By Proposition~\ref{prop:pst-op} the operator $G(-)\iota\colon[F,\,1]\to[I,\,G]$ induces an operator $(-)^\flat\colon[I,\,G]\to[F,\,1]$. Then, by \Cref{rmk:pst-op} we get an operator
\[(-)^\#:=(-)^\flat\cdot(1,F)\colon [I,\,GF]\rightarrow[F,\,F].\] 
We define the extension operator $(-)^\dagger$ of $S=_\textrm{def}GF$ as $(-)^\#$ composed after with $G(-)$ (by Remark~\ref{rmk:pst-op} we get an operator of the required type). As unit we consider $s:=\iota$.  

The left unital law follows from the fact that $(-)^\#$ is the inverse of $G(-)\iota$, and therefore for any $f\colon IA\rightarrow SB$ we have $f=G(f^\#)\iota=f^\dagger\iota$. 

Moreover we can deduce also the right unital law, as
\begin{align*}
(s)^\dagger
&  = G(\,(s)^\#\,)
& (\textrm{by definition of}\;(-)^\dagger) 
\\
& = G(\,(G(1_{F})\iota)^\#\,)
& (\textrm{by definition of}\;s\;\textrm{and \yellow identity \black law in}\,\catK)
\\
& = G1_{F}=1_{GF}=1_{S}
& (\textrm{since}\;(-)^\#\,\textrm{inverse of}\;G(-)\iota). 
\end{align*}
We have left to check the associativity law. Given $f\colon IA\rightarrow SB$ and $g\colon IB\rightarrow SC$, we have
\begin{align*}
(g^\dagger f)^\dagger
&  = G(\,(G(g^\#)f)^\#\,)
& (\textrm{by definition of}\;(-)^\dagger) 
& \\
& = G(\,(G(g^\#)G(f^\#)\iota)^\#\,)
& (\textrm{by left unital law})
& \\
& = G(\,(G(g^\#f^\#)\iota)^\#\,)
& (\textrm{by strict functoriality of}\;G) 
& \\
& = G(g^\#f^\#) = G(g^\#)G(f^\#) =g^\dagger f^\dagger
& (\textrm{by left unital law}). 
& \qedhere
\end{align*} 
\end{proof}


\section{The 2-category of Relative Monads}
\label{sec:Rel(K)}

In this section, fixed a 2-category $\catK$, we will give the definition of the 2-category~$\Rel(\catK)$ of relative monads in $\catK$. 

\begin{defn}
\label{def:rel-mnd-morph}
Let $(X,\,I,\,S)$ and $(Y,\,J,\,T)$ be two relative monads in $\catK$. A \emph{relative monad morphism} $(F,\,F_0,\,\phi)\colon (X,\,I,\,S)\rightarrow(Y,\,J,\,T)$ consists of two 1-cells $F_0\colon X_0\rightarrow Y_0$ and $F\colon X\rightarrow Y$ and a 2-cell $\phi\colon FS\rightarrow TF_0$ satisfying the following axioms:
\begin{itemize}
\item $FI=JF_0$;
\item \emph{unit law}, i.e.\ the following diagram commutes \\
\begin{tikzpicture}
\node (IX) at (0,1) {$FI$};
\node (SX) at (2,2) {$FS$};
\node (SY) at (2,0) {$TF_0$;};
\draw[->] (IX) to node[scale=.7] [above, xshift=-0.1cm]{$Fs$} (SX);
\draw[->] (SX) to node[scale=.7] [right]{$\phi$} (SY);
\draw[->] (IX) to  node[scale=.7] (G'f) [left, yshift=-0.2cm, xshift=0.1cm] {$tF_0$} (SY);
\end{tikzpicture}
\item \emph{extension law}, i.e.\ for any 1-cells \yellow$A,\,B\colon O\rightarrow X_0$ \black and 2-cell $k\colon IA\rightarrow SB$ the following diagram commutes \\
\begin{tikzpicture}
\node (a) at (0,2) {$FSA$};
\node (b) at (2.5,2) {$TF_0A$};
\node (c) at (0,0) {$FSB$};
\node (d) at (2.5,0) {$TF_0B$.};
\draw[->] (a) to node[scale=.7] [above]{$\phi A$} (b);
\draw[->] (c) to node[scale=.7] [below]{$\phi B$} (d);
\draw[->] (a) to node[scale=.7] [left]{$F(k^\dagger_S)$} (c);
\draw[->] (b) to node[scale=.7] [right]{$(\phi B\cdot Fk)^\dagger_T$} (d);
\end{tikzpicture}
\end{itemize}
\end{defn}

\begin{rmk}
\label{rmk:1-cells}
A relative monad morphism $(F,\,F_0,\,\phi)\colon (X,\,I,\,S)\rightarrow(Y,\,J,\,T)$ between monads, i.e.\ when~$X_0=X$, $I=1_{X}$, $Y_0=Y$ and $J=1_{Y}$, is the same as a monad morphism~$(F,\,\phi)\colon (X,\,S)\rightarrow(Y,\,T)$ in $\Mnd(\catK^\op)^\op$. 
\end{rmk}

Two notions of morphisms of relative monads in $\Cat$ appear in \cite[Definition~2.2]{AltChap:mnd-no-end} and \cite[Definition~6]{Ahr:int-typ-synt}. In \cite{AltChap:mnd-no-end} they define a morphism of relative monads between relative monads on a common functor $\catC_0\to\catC$, which is the same as \Cref{def:rel-mnd-morph} (in $\Cat$) setting $I=J$ and $F_0=1_{C_0}$. Instead, in \cite{Ahr:int-typ-synt}, they generalise the definition of colax morphism of monads to relative monads. Our definition of relative monad morphism in $\catK=\Cat$ is a particular case of the one in \cite{Ahr:int-typ-synt} (with $N=1$). We choose to impose the equality $FI=JF_0$ in \Cref{def:rel-mnd-morph} to get exactly monad morphisms when we restrict to relative monad morphisms between monads (see \Cref{rmk:1-cells}).  

\begin{defn}
Let $(F,\,F_0,\,\phi),\,(F',\,F'_0,\,\phi')\colon (X,\,I,\,S)\rightarrow(Y,\,J,\,T)$ be two relative monad morphisms. A \emph{relative monad transformation} $(p,\,p_0)\colon (F,\,F_0,\,\phi)\rightarrow(F',\,F'_0,\,\phi')$ consists of two 2-cells $p\colon F\rightarrow F'$ and $p_0\colon F_0\rightarrow F'_0$ such that:
\begin{itemize}
\item $Jp_0=pI$;
\item the following diagram commutes \\
\begin{tikzpicture}
\node (IX) at (0,2) {$FS$};
\node (TF0) at (0,0) {$TF_0$};
\node (SX) at (2,2) {$F'S$};
\node (SY) at (2,0) {$TF'_0$.};
\draw[->] (IX) to node[scale=.7] [above]{$pS$} (SX);
\draw[->] (IX) to node[scale=.7] [left]{$\phi$} (TF0);
\draw[->] (SX) to node[scale=.7] [right]{$\phi'$} (SY);
\draw[->] (TF0) to  node[scale=.7] (G'f) [below] {$Tp_0$} (SY);
\end{tikzpicture}
\end{itemize}
\end{defn}

\begin{rmk}
\label{rmk:2-cells}
A relative monad transformation $(p,\,p_0)\colon (F,\,F_0,\,\phi)\rightarrow(F',\,F'_0,\,\phi')$ with $X_0=X$, $I=1_{X}$, $Y_0=Y$ and $J=1_{Y}$ is the same as a monad transformation of the form $p\colon (F,\,\phi)\rightarrow(F',\,\phi')$ in the sense of Street \cite{StreetR:fortm}. 
\end{rmk}

\begin{prop}
\label{prop:mnd-subcat-rel}
Let $\catK$ be a 2-category. There is a 2-category $\Rel(\catK)$ of relative monads in $\catK$ with relative monads, relative monad morphisms and relative monad transformations as 0-, 1- and 2-cells.
\qed
\end{prop} 

Using part (\ref{rmk:0-cells}) of Example~\ref{eg:rel-mnd} and Remarks \ref{rmk:1-cells} and \ref{rmk:2-cells} we get the following proposition, which shows how our definition of $\Rel(\catK)$ extends Street's definition of $\Mnd(\catK)$ \cite{StreetR:fortm}, the 2-category of monads in a 2-category $\catK$. Before stating the proposition we recall the definition of \emph{full sub-2-category}. 

Let $\catK$ and $\catL$ be two 2-categories. A 2-functor $J\colon\catK\to\catL$ exhibits $\catK$ as a \emph{full sub-2-category} of $\catL$ if for all pair of objects $x,\,y\in\catK$ the functor $J_{x,y}\colon\catK(x,\,y)\to\catL(Jx,\,Jy)$ is an equivalence of categories.

\begin{prop}
$\Mnd(\catK^\op)^\op$ is a \yellow full \black sub-2-category of $\Rel(\catK)$ consisting of relative monads $(X,\,I,\,S)$ with $X_0=X$ and $I=1_X$. 
\end{prop}

In Section~\ref{sec:rrm-kl-ext} we will describe an equivalent way to define morphisms of relative monads using a generalised version of right modules. We will then build a 2-category 2-isomorphic to $\Rel(\catK)$. 

\section{Relative Algebras}
\label{sec:rel-alg}

We now introduce the notion of an Eilenberg--Moore object for a relative monad, which we will refer to as \emph{relative EM object}. The approach used is the same as the one in \cite{StreetR:fortm}. The notion of \emph{relative} algebra for a relative monad has been already introduced in \cite{AltChap:mnd-no-end, MarmWood:ext-syst}, here we complete it using our definition of operator (Definition~\ref{defn:operator}). From now on we will consider a relative monad $(X,\,I,\,T)\in\Rel(\catK)$. 

\begin{defn}
Let $K\in\catK$. A \emph{$K$-indexed relative EM-algebra} (or \emph{relative left module}) consists of a 1-cell $M\colon K\rightarrow X$ together with an operator $(-)^m\colon [I,\,M]\rightarrow[T,\,M]$ satisfying the following axioms:
\begin{itemize}
\item \emph{unit law}, i.e.\ for any span $(A,\,B):O\rightarrow X_0;K$ and any 2-cell $h\colon IA\rightarrow MB$ the diagram below commutes\\
\begin{tikzpicture}
\node (IX) at (0,1.5) {$IA$};
\node (SX) at (1.5,1.5) {$TA$};
\node (SY) at (1.5,0) {$MB$;};
\draw[->] (IX) to node[scale=.7] [above]{$tA$} (SX);
\draw[->] (SX) to node[scale=.7] [right]{$h^m$} (SY);
\draw[->] (IX) to [bend right]  node[scale=.7] (G'f) [left, xshift=-0.1cm] {$h$} (SY);
\end{tikzpicture}

\item \emph{associativity law}, i.e.\ for any pair of spans $(A,\,B)\colon O\rightarrow X_0;K$ and \yellow $(C,\,A)\colon O\rightarrow X_0$ \black and any 2-cells $h\colon IA\rightarrow MB$ and $k\colon IC\rightarrow TA$, the diagram below commutes \\
\begin{tikzpicture}
\node (IX') at (0,1.5) {$TC$};
\node (SX') at (1.5,1.5) {$TA$};
\node (SY') at (1.5,0) {$MB$.};
\draw[->] (IX') to node[scale=.7] [above]{$k^\dagger$} (SX');
\draw[->] (SX') to node[scale=.7] [right]{$h^m$} (SY');
\draw[->] (IX') to [bend right]  node[scale=.7] (G'f) [left, xshift=-0.1cm] {$(h^m\cdot k)^m$} (SY');
\end{tikzpicture}
\end{itemize}
\end{defn}

\begin{defn}
Let $K\in\catK$ and let $(M,\,(-)^m)$ and $(N,\,(-)^n)$ be two $K$-indexed relative EM-algebras. A \emph{$K$-indexed relative EM-algebra morphism} $f\colon (M,\,(-)^m)\rightarrow (N,\,(-)^n)$ consists of a 2-cell $f\colon M\rightarrow N$ such that for any 2-cell $h\colon IA\rightarrow MB$ (given any span $(A,\,B)\colon O\rightarrow X_0;K$):
\begin{center}
\begin{tikzpicture}
\node (IX) at (0,2) {$TA$};
\node (SX) at (2,2) {$MB$};
\node (SY) at (2,0) {$NB$.};
\draw[->] (IX) to node[scale=.7] [above]{$h^m$} (SX);
\draw[->] (SX) to node[scale=.7] [right]{$fB$} (SY);
\draw[->] (IX) to [bend right]  node[scale=.7] (G'f) [left, xshift=-0.1cm] {$(fB\cdot h)^n$} (SY);
\end{tikzpicture}
\end{center}
\end{defn}

Clearly there is a category $\REM^T(K)$ of $K$-indexed relative algebras. Therefore we have an induced 2-functor 
\begin{center}
\begin{tikzpicture}
\node (catk) at (0,2.25) {$\catK^\op$};
\node (Cat) at (3,2.25) {$\Cat$};
\node (K') at (0,0) {$K'$};
\node (K) at (0,1.5) {$K$};
\node (ModK) at (3,1.5) {$\REM^T(K)$};
\node (ModK') at (3,0) {$\REM^T(K')$};
\node (MF) at (6.5,0) {$(\,MF,(-)_m\cdot(1,F)\,)$,};
\node (M) at (6.5,1.5) {$(\,M,(-)_m\,)$};
\draw[->] (catk) to node[scale=.9] [above]{$\REM^T(-)$} (Cat);
\draw[|->] (K) to (ModK);
\draw[|->] (K') to (ModK');
\draw[->] (K') to  node[scale=.9] [left]{$F$} (K);
\draw[->] (ModK) to  node[scale=.9] [right]{$-\circ F$} (ModK');
\draw[|->] (M) to (MF);
\end{tikzpicture}
\end{center}
where we recall that the notation $(-)_m\cdot (1,F)$ was defined in \Cref{rmk:pst-op}.
\begin{defn}
We say that a relative monad $(X,\,I,\,T)\in\Rel(\catK)$ has a \emph{relative EM object} if $\REM^T(-)\colon \catK^\op\rightarrow\Cat$ is representable. We will denote the representing object with $X^{I,T}$. 
\end{defn}
\newpage
\begin{eg}
\label{eg:EM-obj-cat}
Let $T\text{-}\mathrm{Alg}$ be the category of EM-algebras for a relative monad $(\catX,\,I,\,T)$ in $\Cat$ defined in \cite{AltChap:mnd-no-end}. We can see that this gives us a relative EM object for $T$. In order to prove it we just need to notice that, for any category~$\mathbb{K}$, a $\mathbb{K}$-indexed relative algebra $M\colon \mathbb{K}\rightarrow\catX$ is the same as endowing, for any~$k\in\mathbb{K}$, each $Mk\in\catX$ with a relative EM-algebra structure. Therefore each relative algebra $M$ induces a functor $\overline{M}\colon \mathbb{K}\rightarrow T\text{-}\mathrm{Alg}$. On the other hand if we start with a functor $\overline{M}$ then its composition with $U\colon T\text{-}\mathrm{Alg}\rightarrow\catX$ has a relative EM-algebra structure. These constructions are clearly inverses of each other and provide a natural isomorphism
\begin{center}
$\REM^T(\mathbb{K})\cong\Cat(\mathbb{K},\,T\text{-}\mathrm{Alg})$.
\end{center}
\end{eg}

For any $(X,\,I,\,T)\in\Rel(\catK)$, we can define~$U\colon \REM^T(-)\rightarrow\catK(-,\,X)$ the forgetful natural transformation, sending a relative \yellow left \black $T$-module to its underlying 1-cell, and the \emph{free relative algebra} transformation $F\colon \catK(-,\,X_0)\rightarrow\REM^T(-)$ defined as, for any indexing object $K\in\catK$, 
\begin{center}
$K\xrightarrow{M}X_0\;\longmapsto (\,TM,\,(-)^\dagger\cdot(1,M)\,)$.
\end{center}
Therefore we get a diagram in $\hat{\catK}$ of the form
\begin{center}
\begin{tikzpicture}
\node (a) at (0,0) {$\catK(-,X_0)$};
\node (b) at (2,1.5) {$\REM^T(-)$};
\node (c) at (4,0) {$\catK(-,X)$.};

\path[->] 
(a) edge node[scale=.7] [left,xshift=-0.1cm,yshift=0.1cm]{$F$} (b)
	edge node[scale=.7] (I) [below]{$I_\ast:=I\circ -$} (c)
(b) edge node[scale=.7] [right, xshift=0.1cm,yshift=0.1cm]{$U$} (c)
;

\draw[-{Implies},double distance=1.5pt,shorten >=10pt,shorten <=10pt] (I) to node[scale=.7] [right, xshift=0.1cm] {$t$} (b);

\end{tikzpicture}
\end{center}

\begin{lemma}
\label{lemma:rel-EM-to-rel-mnd}
The natural transformations defined above form a relative adjunction $F {\,\,}_{I_\ast}\!\dashv U$ in $\hat{\catK}$. Moreover, the relative monad induced by it is $(\catK(-,\,X),\,I_\ast,\,T_\ast)$. 
\end{lemma}

\begin{proof}
In order to prove the first claim we need to find, for any $M\in\catK(K,\,X_0)$ and $(N,\,(-)^n)\in\REM^T(K)$, a natural bijection
\begin{center}
$\REM^T(K)(FM,\,(N,\,(-)^n))\cong\catK[K,\,X](IM,\,N)$.
\end{center}
For any relative algebra map $\bar{f}\colon FM=(TM,\,(-)^\dagger)\rightarrow(N,\,(-)^n)$ we define the 2-cell~$\bar{f}^\#\colon IM\rightarrow N$ as $\bar{f}\cdot tM$. On the other hand for any 2-cell $f\colon IM\rightarrow N$ we can define $f^\flat:=f^n\colon TM\rightarrow N$ which is a relative algebras map because, for any $A,B\colon O\rightarrow X_0$ and $h\colon IA\rightarrow TMB$, 
\begin{center}
\begin{tikzpicture}
\node (IX) at (0,2) {$TA$};
\node (SX) at (2,2) {$TMB$};
\node (SY) at (2,0) {$NB$};
\draw[->] (IX) to node[scale=.7] [above]{$h^\dagger$} (SX);
\draw[->] (SX) to node[scale=.7] [right]{$f^nB=(fB)^n$} (SY);
\draw[->] (IX) to [bend right]  node[scale=.7] (G'f) [left, xshift=-0.1cm] {$(f^nB\cdot h)^n$} (SY);
\end{tikzpicture}
\end{center}
which is true since $(-)^n$ is a relative algebra operator. 

For any $f$ we can see that $(f^\flat)^\#=f^n\cdot tM$ which is exactly equal to $f$ by the unit law of $(-)^n$. Moreover, for any $\bar{f}$, we have 
\begin{align*}
(\bar{f}^\#)^\flat
&   =(\bar{f}\cdot tM)^n
& (\textrm{by definitions}) 
\\
& = \bar{f}\cdot(tM)^\dagger
& (\textrm{since}\;\bar{f}\;\textrm{is in}\,\REM^T(K))
\\
& = \bar{f}\cdot1_{TM}=\bar{f}
& (\textrm{by right unital law of}\;(-)^\dagger). 
\end{align*}
We can see that $UF=T_\ast$ and more generally the relative monad induced by $F {\,}_{I_\ast}\dashv U$ is the same as the one induced by $(X,\,I,\,T)$ in $\hat{\catK}$. 
\end{proof}

\begin{theorem}
\label{thm:EM-obj}
If $\catK$ has relative EM objects, then any relative monad is induced by a relative adjunction. 
\end{theorem}

\begin{proof}
The proof is just a matter of translating Lemma~\ref{lemma:rel-EM-to-rel-mnd} using the Yoneda lemma, since the covariant Yoneda embedding reflects adjunctions \cite[Proposition~I,6.4]{Gray:Adj} and also relative adjunctions. Explicitly, the relative adjunction that we get is
\[\begin{tikzcd}
	& {X^{I,T}} \\
	{X_0} && X
	\arrow["{J}", from=2-1, to=1-2]
	\arrow["{U^{I,T}}", from=1-2, to=2-3]
	\arrow[""{name=0, anchor=center, inner sep=0}, "I"', from=2-1, to=2-3]
	\arrow["t"', shorten <=6pt, shorten >=6pt, Rightarrow, from=0, to=1-2]
\end{tikzcd}\]
where $U^{I,T}\colon X^{I,T}\to X$ is the $X^{I,T}$-indexed relative EM-algebra corresponding to $1_{X^{I,T}}$ and $J_\ast=F$, $(U^{I,T})_\ast=U$ and $U^{I,T}J=T$. 
\end{proof}

\section{Relative Distributive Laws}
\label{sec:rel-distr-law}

In this section, we will define the counterpart of distributive laws for a relative monad $I,T\colon X_0\rightarrow X$ and a monad $S\colon X\rightarrow X$ which \emph{restricts to $X_0$}. When $I$ is an inclusion it is clear what we mean by this, but when $I$ is any 1-cell we need to define a new notion. Therefore, with the following definition, we introduce the notion of \emph{monad compatible} with a 1-cell $I:X_0\to X$.

\begin{defn}
\label{defn:comp-mnd}
Let $I\colon X_0\rightarrow X$ be 1-cell in a 2-category $\catK$. A \emph{monad compatible with~$I$} consists of a pair of monads $(X_0,\,S_0)$ and $(X,\,S)$ in $\catK$ such that $SI=IS_0$, $mI=Im_0$ and $sI=Is_0$. We will denote it with $\compS$. 
\end{defn}

To have a monad compatible with $I$ is the same as lifting $I$ to a morphism in~$\Mnd(\catK)$ with corresponding 2-cell the identity, i.e.\ requiring $(I,\,1)\colon (X_0,\,S_0)\rightarrow(X,\,S)$ to be a monad morphism. 
\newpage
\begin{defn}
\label{defn:rel-distr}
Let $(X,\,I,\,T)$ be a relative monad in $\catK$ and $(S,\,S_0)$ a monad compatible with $I$. A \emph{relative distributive law} of $T$ over $\compS$ consists of a 2-cell $d\colon ST\rightarrow TS_0$ in $\catK$ satisfying the following axioms:
\begin{center}
\begin{tikzpicture}
\node at (-2.5,2) {(D1)};
\node (S2T) at (-0.5,4) {$S^2T$};
\node (STS0) at (-0.5,2) {$STS_0$};
\node (TS02) at (-0.5,0) {$TS_0^2$};
\node (ST) at (1.5,4) {$ST$};
\node (TS0) at (1.5,0) {$TS_0$};
\path[->] 
(S2T) edge node[scale=.7] [above] {$mT$} (ST)
	edge node[scale=.7] [left] {$Sd$} (STS0)
(ST) edge node[scale=.7] [right] {$d$} (TS0)
(STS0) edge node[scale=.7] [left] {$dS_0$} (TS02)
(TS02) edge node[scale=.7] [below] {$Tm_0$} (TS0);

\node at (5,2) {(D2)};
\node (T) at (6.5,2) {$T$};
\node (ST') at (9,3) {$ST$};
\node (TS0') at (9,1) {$TS_0$};
\path[->] 
(ST') edge node[scale=.7] [right] {$d$} (TS0')
(T) edge node[scale=.7] [above] {$sT$} (ST')
	edge node[scale=.7] [below] {$Ts_0$} (TS0');
\end{tikzpicture}
\end{center}
and for any object $O\in\catK$, any pair of 1-cells $A,\,B\colon O\rightarrow X_0$ and any 2-cell $f\colon IA\rightarrow TB$
\begin{center}
\begin{tikzpicture}
\node at (-2,1) {(D3)};
\node (S2T) at (0,2) {$STA$};
\node (STS0) at (0,0) {$STB$};
\node (ST) at (2,2) {$TS_0A$};
\node (TS0) at (2,0) {$TS_0B$};
\path[->] 
(S2T) edge node[scale=.7] [above] {$dA$} (ST)
	edge node[scale=.7] [left] {$Sf^\dagger_T$} (STS0)
(ST) edge node[scale=.7] [right] {$(dB\cdot Sf)^\dagger_T$} (TS0)
(STS0) edge node[scale=.7] [below] {$dB$} (TS0);

\node at (5.5,1) {(D4)};
\node (T) at (7,2) {$SI$};
\node (T') at (7,0) {$IS_0$};
\node (ST') at (9.5,2) {$ST$};
\node (TS0') at (9.5,0) {$TS_0$.};
\path[->] 
(ST') edge node[scale=.7] [right] {$d$} (TS0')
(T) edge node[scale=.7] [above, yshift=0.1cm] {$St$} (ST')
(T') edge node[scale=.7] [below, yshift=-0.1cm] {$tS_0$} (TS0');
\draw [double equal sign distance] (T) to [out=-90, in=90] (T');
\end{tikzpicture}
\end{center}
\end{defn}

From now on, we will always consider a relative monad $(X,\,I,\,T)$ and a monad $(S,\,S_0)$ compatible with $I$. 

\begin{rmk}
We can see that, setting $I=1_X$, we get back the definition of a distributive law between two monads $T$ and $S$ in $\catK$ \cite{ManesE:mnd-comp}.
\end{rmk}

In the formal theory of monads \cite{StreetR:fortm} Street shows that a distributive law between two monads is an object of $\Mnd(\Mnd(\catK))$, that is a monad in the 2-category of monads. The next Proposition proves a similar result for relative distributive laws. 

\begin{prop}
\label{prop:rel-distr-mnd(rel)}
The objects of $\Mnd(\Rel(\catK))$ are exactly relative distributive law.
\end{prop}

\begin{proof}
Let us unravel what is a monad in the 2-category $\Rel(\catK)$. As data we have a 1-cell $(S,S_0,d)\colon (X,I,T)\to(X,I,T)$ and 2-cells $(m,m_0)\colon (S,S_0,d)^2\to(S,S_0,d)$ and $(s,s_0)\colon(1_{X},1_{X_0},1_T)\to(S,S_0,d)$. Hence, we get the data for a relative monad $(X,I,T)$ and two monads $S_0\colon X_0\to X_0$ and $S\colon X\to X$.
The table below provides a correspondence between axioms.
\begin{center}
\renewcommand\arraystretch{1.3}\begin{tabular}{|c|c|r}
\cline{1-2}
Axiom & In  $\Mnd(\Rel(\catK))$ &\\
\cline{1-2}
(D1) and $mI=Im_0$ & $(m,\,m_0)$ is a 2-cell in  $\Rel(\catK)$ &\\ 
(D2) and $sI=Is_0$ & $(s,\,s_0)$ is a 2-cell in  $\Rel(\catK)$ &\\ 
(D3), (D4) and $SI=IS_0$ & $(S,\,S_0,\,d)$ is a 1-cell in $\Rel(\catK)$ & \qedhere\\
\cline{1-2}
\end{tabular}
\end{center}\qedhere
\end{proof}

The aim of the last part of this section is to prove a Beck-like theorem for relative distributive laws, using liftings to the algebras of a monad $(S,\,S_0)$ compatible with $I$. First of all, we will explicitly define a lifting of a relative monad $T$ to the algebras of $(S,\,S_0)$. Then we will show how we can go  from relative distributive laws to liftings (Lemma~\ref{lemma:distr-to-lift}) and vice versa (Theorem\til\ref{thm:lift-to-distr}). Finally, we show that these constructions are inverses of each other.  

Before proceeding with the definition of lifting to algebras, let us fix some notation. Given a monad $(S,\,S_0)$ compatible with $I\colon X_0\rightarrow X$, we always get a natural transformation induced on indexed algebras, $I_\ast\colon \Soalg(-)\rightarrow\Salg (-)$, defined, for any indexing object $K\in\catK$ and any $S_0$-algebra $(M,\,m)$, as 
\begin{center}
$(S_0M\xrightarrow{m} M)\longmapsto(SIM\xrightarrow{Im} IM)$,
\end{center}
where $Im\colon SIM\rightarrow IM$ is well-defined because $SI=IS_0$. Let us denote with $U_0$ and~$U$ the forgetful natural transformations from $\Soalg(-)$ and $\Salg(-)$ into $\catK(-,X_0)$ and~$\catK(-,X)$. 

\begin{defn}
\label{defn:lift-rel-alg}
Let $(X,\,I,\,T)$ be a relative monad in $\catK$. A \emph{lifting of $T$ to the algebras of $(S,\,S_0)$} is a relative monad $(I_\ast,\,\widehat{T})\colon \Soalg(-)\longrightarrow\Salg(-)$ \yellow in $\hat{\catK}$\black, such that: 
\begin{enumerate}[(i)]
\item \label{lift-rel-alg:comp}the following diagram commutes
\begin{center}
\begin{tikzpicture}
\node (a) at (0,0) {$\catK(-,X_0)$};
\node (b) at (3,2) {$\Salg(-)$};
\node (c) at (3,0) {$\catK(-,X)$;};
\node (d) at (0,2) {$\Soalg(-)$};

\path[->] 
(d) edge node[scale=.7] (F) [above]{$\widehat{T}$} (b)
	edge node[scale=.7] [left,xshift=-0.1cm,yshift=0.1cm]{$U_0$} (a)
(a) edge node[scale=.7] (I) [below]{$T\circ -$} (c)
(b) edge node[scale=.7] [right, xshift=0.1cm,yshift=0.1cm]{$U$} (c)
;
\end{tikzpicture}
\end{center}

\item  \label{lift-rel-alg:ext}the extension operator $(-)^\dagger_{\widehat{T}}$ of $\widehat{T}$ is induced by the one of $T$, i.e.\ for any pair of $K$-indexed $S_0$-algebras $(M,\,m)$ and $(N,\,n)$ if the following diagram on the left commutes, then the one on the right commutes as well
\begin{center}
\begin{tikzpicture}
\node (T) at (-0.5,2) {$IS_0M\yellow{=SIM}$};
\node (T') at (2.2,2) {$STN$};
\node (ST') at (-0.5,0) {$IM$};
\node (TS0') at (2.2,0) {$TN$};
\path[->] 
(ST') edge node[scale=.7] [below] {$f$} (TS0')
(T) edge node[scale=.7] [left] {$Im$} (ST')
	edge node[scale=.7] [above] {\yellow{$Sf$}} (T')
(T') edge node[scale=.7] [right] {$\widehat{T}(n)$} (TS0');

\node at (3.7,1) {$\Rightarrow$};

\node (T) at (5.2,2) {$STM$};
\node (T') at (7.4,2) {$STN$};
\node (ST') at (5.2,0) {$TM$};
\node (TS0') at (7.4,0) {$TN$;};
\path[->] 
(ST') edge node[scale=.7] [below] {$f^\dagger$} (TS0')
(T) edge node[scale=.7] [left] {$\widehat{T}(m)$} (ST')
	edge node[scale=.7] [above] {$Sf^\dagger$} (T')
(T') edge node[scale=.7] [right] {$\widehat{T}(n)$} (TS0');
\end{tikzpicture}
\end{center} 
\item \label{lift-rel-alg:unit}the unit $\widehat{t}$ is induced by $t$, i.e.\ for any $S_0$-algebra $(M,\,m)$ the following diagram commutes
\begin{center}
\begin{tikzpicture}
\node (T) at (0,2) {$IS_0M$};
\node (T') at (2.2,2) {$STM$};
\node (ST') at (0,0) {$IM$};
\node (TS0') at (2.2,0) {$TM$.};
\path[->] 
(ST') edge node[scale=.7] [below] {$tM$} (TS0')
(T) edge node[scale=.7] [left] {$Im$} (ST')
	edge node[scale=.7] [above] {$StM$} (T')
(T') edge node[scale=.7] [right] {$\widehat{T}m$} (TS0');
\end{tikzpicture}
\end{center}
\end{enumerate}
\end{defn}

\yellow The next proposition gives a description of a lifting $\widehat{T}$ in terms of its action on free algebras. This result will be useful in \Cref{thm:lift-to-distr,thm:rel-distr-eq-lift}.\black 

\begin{prop}
\label{prop:lift-on-struct}
Let $\widehat{T}$ be a lifting of $T$ to the algebras of $\compS$ and let us denote with $\widehat{{m_0}M}$ the $S$-algebra structure on $TS_0M$ given by $\widehat{T}$ applied to the free $S_0$-algebra $(S_0M,\,{m_0}M)$. Then, for any other $S_0$-algebra $(M,\,m)$ the $S$-algebra structure on $TM$ given by $\widehat{T}$ is 
\begin{center}
$STM\xrightarrow{ ST{s_0}M}STS_0M\xrightarrow{\widehat{{m_0}M}}TS_0M\xrightarrow{Tm}TM$.
\end{center}
\end{prop}

\begin{proof}
We begin noticing that, by one of the algebra axioms, $m$ is itself a $S_0$-algebra morphism  between $(S_0M,\,{m_0}M)$ and $(M,\,m)$. Therefore the diagram below commutes, as it is the diagram making $Tm$ a $S$-algebra morphism, 
\begin{center}
\begin{tikzpicture}
\node (T) at (0,2) {$STS_0M$};
\node (T') at (2.2,2) {$STM$};
\node (ST') at (0,0) {$TS_0M$};
\node (TS0') at (2.2,0) {$TM$.};
\path[->] 
(ST') edge node[scale=.7] [below] {$Tm$} (TS0')
(T) edge node[scale=.7] [left] {$\widehat{{m_0}M}$} (ST')
	edge node[scale=.7] [above] {$STm$} (T')
(T') edge node[scale=.7] [right] {$\widehat{T}m$} (TS0');
\end{tikzpicture}
\end{center}
Moreover, using the unit algebra axiom for $m$, we get the desired equality
\begin{center}
\begin{tikzpicture}
\node (STX) at (-2.5,2) {$STM$};
\node (T) at (0,2) {$STS_0M$};
\node (ST') at (2.5,2) {$STM$};
\node (T') at (0,0) {$TS_0M$};
\node (TS0') at (2.5,0) {$TM$.};

\path[->] 
(ST') edge node[scale=.7] [right] {$Tm$} (TS0')
(T) edge node[scale=.7] [above] {$\widehat{{m_0}M}$} (ST')
	edge node[scale=.7] [left] {$STm$} (T')
(T') edge node[scale=.7] [below] {$\widehat{T}m$} (TS0')
(STX) edge node[scale=.7] [above] {$ST{s_0}M$} (T)
	edge[bend right] node[scale=.7] [below, xshift=-1.2cm] {$1_{STM}=S1_{TM}$} (T');
\end{tikzpicture} \\
\hfill \qedhere
\end{center} 
\end{proof}

\begin{lemma}
\label{lemma:distr-to-lift}
Let $d\colon ST\rightarrow TS_0$ be a relative distributive law of $T$ over $\compS$. Then there is a lifting $\widehat{T}$ of $T$ to the algebras of $\compS$ defined on $K$-indexed $S_0$-algebras $(M,\,m)$ as  
$$STM\xrightarrow{dM}TS_0M\xrightarrow{Tm}TM$$
and on morphisms of $K$-indexed $S_0$-algebras $f\colon (M,\,m)\rightarrow(N,\,n)$ as
\begin{center}
$Tf\colon (TM,\,Tm\cdot dM)\rightarrow (TN,\,Tn\cdot dN)$. 
\end{center}
\end{lemma}

\begin{proof}
First of all, we need to verify that the definition above gives a $K$-indexed $S$-algebra structure. One can check this using (D1) for the compatibility axiom and (D2) for the unit. Now we have left to prove part (\ref{lift-rel-alg:ext}) and (\ref{lift-rel-alg:unit}) of the definition of a lifting. For the first one what we need to check is the following implication
\begin{center}
\begin{tikzpicture}
\node (T) at (-0.2,2.5) {$IS_0M$};
\node (T') at (2.2,2.5) {$STN$};
\node (TS0Y) at (2.2,1) {$TS_0N$};
\node (ST') at (-0.2,-0.5) {$IM$};
\node (TS0') at (2.2,-0.5) {$TN$};
\path[->] 
(ST') edge node[scale=.7] [below] {$f$} (TS0')
(T) edge node[scale=.7] [left] {$Im$} (ST')
	edge node[scale=.7] [above] {$Sf$} (T')
(T') edge node[scale=.7] [right] {$dN$} (TS0Y)
(TS0Y) edge node[scale=.7] [right] {$Tn$} (TS0');

\node at (3.7,1) {$\Rightarrow$};

\node (T) at (5.2,2.5) {$STM$};
\node (T') at (7.6,2.5) {$STN$};
\node (TS0X) at (5.2,1) {$TS_0M$};
\node (TS0Y) at (7.6,1) {$TS_0N$};
\node (ST') at (5.2,-0.5) {$TM$};
\node (TS0') at (7.6,-0.5) {$TN$.};
\path[->] 
(ST') edge node[scale=.7] [below] {$f^\dagger$} (TS0')
(T) edge node[scale=.7] [left] {$dM$} (TS0X)
	edge node[scale=.7] [above] {$Sf^\dagger$} (T')
(TS0X) edge node[scale=.7] [left] {$Tm$} (ST')
	edge[dotted] node[scale=.7] [above] {$(dN\cdot Sf)^\dagger$} (TS0Y)
(T') edge node[scale=.7] [right] {$dN$} (TS0Y)
(TS0Y) edge node[scale=.7] [right] {$Tn$} (TS0');
\end{tikzpicture}
\end{center}
Using (D3) is enough to prove that the bottom square on the right commutes whenever the diagram on the left does. 
\begin{align*}
Tn\cdot(dN\cdot Sf)^\dagger 
&  = (Tn\cdot dN\cdot Sf)^\dagger
& (\textrm{by naturality of}\;(-)^\dagger) 
\\
& = (f\cdot Im)^\dagger
& (\textrm{by diagram on the left}) 
\\
& =  f^\dagger\cdot Tm
& (\textrm{by naturality of}\;(-)^\dagger)
\end{align*}
Similarly part (iii) follows from (D4) and the naturality of $t$. Indeed, for any $K$-indexed $S_0$-algebra $(M,\,m)$ the following diagram is commutative
\begin{center}
\begin{tikzpicture}
\node (T) at (-0.2,2.5) {$IS_0M$};
\node (T') at (2.2,2.5) {$STM$};
\node (TS0Y) at (2.2,1) {$TS_0M$};
\node (ST') at (-0.2,-0.5) {$IM$};
\node (TS0') at (2.2,-0.5) {$TM$.};
\path[->] 
(ST') edge node[scale=.7] [below] {$tM$} (TS0')
(T) edge node[scale=.7] [left] {$Im$} (ST')
	edge node[scale=.7] [above] {$StM$} (T')
	edge node[scale=.7] [below, yshift=-0.2cm] {$t{S_0M}$} (TS0Y)
(T') edge node[scale=.7] [right] {$dM$} (TS0Y)
(TS0Y) edge node[scale=.7] [right] {$Tm$} (TS0');
\end{tikzpicture} \\
\hfill \qedhere
\end{center}
\end{proof}

\begin{theorem}
\label{thm:lift-to-distr}
Let $\widehat{T}$ be a lifting of $T$ to the algebras of $\compS$. Then $d$ defined as 
$$ST\xrightarrow{ ST{s_0}}STS_0\xrightarrow{\widehat{{m_0}}}TS_0$$
where $\widehat{{m_0}}$ is the $X_0$-indexed $S$-algebra structure of $\widehat{T}(S_0,\,m_0)$, is a relative distributive law of $T$ over $\compS$.
\end{theorem}

\begin{proof}
We need to prove that axioms (D1), (D2), (D3) and (D4) hold. In the following table we explain what will be used to prove each axiom. 
\begin{center}
\renewcommand\arraystretch{1.3}\begin{tabular}{|c|c|}
\hline
Axiom & Axioms used in the proof \\
\hline 
(D1) & $S$-algebra axiom for $\widehat{{m_0}_X}$ and Proposition~\ref{prop:lift-on-struct} \\ 
(D2) & unit algebra axiom for $\widehat{{m_0}_X}$ \\ 
(D3) & (D4), (D1), part (\ref{lift-rel-alg:ext}) of Definition~\ref{defn:lift-rel-alg} and more \\ 
(D4) & part (\ref{lift-rel-alg:unit}) Definition~\ref{defn:lift-rel-alg} \\
\hline
\end{tabular}
\end{center}
For (D1), (D2) and (D4) it is enough to write down what we get explicitly using the definition of $d$. The diagrams we get are the following

\begin{center}
\begin{tikzpicture}
\node at (-2,2) {(D1)};
\node (S2T) at (0,4) {$S^2T$};
\node (S2TS0) at (3,4) {$S^2TS_0$};
\node (STS0u) at (6,4) {$STS_0$};
\node (STS02) at (9,4) {$STS_0^2$};
\node (TS02) at (9,2) {$TS_0^2$};
\node (ST) at (0,0) {$ST$};
\node (STS0d) at (3,0) {$STS_0$};
\node (TS0) at (9,0) {$TS_0$};

\path[->] 
(S2T) edge node[scale=.7] [above] {$S^2T{s_0}$} (S2TS0)
	edge node[scale=.7] [left] {$m{T}$} (ST)
(S2TS0) edge node[scale=.7] [above] {$S\widehat{{m_0}}$} (STS0u)
	edge node[scale=.7] [left] {$m{TS_0}$} (STS0d)
(STS0u) edge node[scale=.7] [above] {$ST{s_0}{S_0}$} (STS02)
	edge node[scale=.7] [left] {$\widehat{{m_0}}$} (TS0)
(STS02) edge node[scale=.7] [right] {$\widehat{{m_0}{S_0}}$} (TS02)
(TS02) edge node[scale=.7] [right] {$T{m_0}$} (TS0)
(ST) edge node[scale=.7] [below] {$ST{s_0}$} (STS0d)
(STS0d) edge node[scale=.7] [below] {$\widehat{{m_0}}$} (TS0)
;
\end{tikzpicture} \\ \vspace{0.2cm}
\begin{tikzpicture}
\node at (-2,2) {(D2)};
\node (T) at (0,4) {$T$};
\node (ST) at (2.5,4) {$ST$};
\node (TS0u) at (0,2) {$TS_0$};
\node (STS0) at (2.5,2) {$STS_0$};
\node (TS0d) at (2.5,0) {$TS_0$};

\path[->] 
(T) edge node[scale=.7] [above] {$s{T}$} (ST)
	edge node[scale=0.7] [left] {$T{s_0}$} (TS0u)
(TS0u) edge node[scale=.7] [above] {$s{TS_0}$} (STS0)
	edge[bend right] node[scale=0.7] [left] {$1_{TS_0}$} (TS0d)
(ST) edge node[scale=0.7] [right] {$ST{s_0}$} (STS0)
(STS0) edge node[scale=0.7] [right] {$\widehat{{m_0}}$} (TS0d)
;

\node at (4.8,2) {\yellow{(D4)}};
\node (SI) at (6.5,4) {$SI$};
\node (ST) at (9,4) {$ST$};
\node (SIS0) at (6.5,2) {$IS_0^2$};
\node (STS0) at (9,2) {$STS_0$};
\node (IS0) at (6.5,0) {$IS_0$};
\node (TS0) at (9,0) {$TS_0$.};

\draw [double equal sign distance] (SI) to [out=-120, in=120] (IS0);
\path[->] 
(SI) edge node[scale=.7] [above] {$St$} (ST)
	edge node[scale=.7] [right] {$SI{s_0}$} (SIS0)
(SIS0) edge node[scale=.7] [above] {$St{S_0}$} (STS0)
	edge node[scale=.7] [right] {$I{m_0}$} (IS0)
(ST) edge node[scale=.7] [right] {$ST{s_0}$} (STS0)
(STS0) edge node[scale=.7] [right] {$\widehat{{m_0}}$} (TS0)
(IS0) edge node[scale=.7] [below] {$t{S_0}$} (TS0)
;
\end{tikzpicture}
\end{center}
Let us now look at (D3). For any 2-cell $\alpha\colon IA\rightarrow TB$ in $\catK$ (for any pair of 1-cells $A,B\colon O\rightarrow X_0$), we need to prove that 
\begin{equation}
\label{eq:lift-diag}
\begin{gathered}
\begin{tikzpicture}
\node (T) at (-0.3,2) {$STA$};
\node (T') at (2.5,2) {$STB$};
\node (ST') at (-0.3,0) {$TS_0A$};
\node (TS0') at (2.5,0) {$TS_0B$};
\path[->] 
(ST') edge node[scale=.7] [below] {$(dB\cdot S\alpha)^\dagger$} (TS0')
(T) edge node[scale=.7] [left] {$dA$} (ST')
	edge node[scale=.7] [above] {$S\alpha^\dagger$} (T')
(T') edge node[scale=.7] [right] {$dB$} (TS0');

\node at (4.5,1) {{i.e.}};

\node at (7.9,1.8) {($\ast$)};
\node at (7.9,0.25) {($\ast\ast$)};
\node (T) at (6.5,2.5) {$STA$};
\node (T') at (9.3,2.5) {$STB$};
\node (TS0X) at (6.5,1) {$STS_0A$};
\node (TS0Y) at (9.3,1) {$STS_0B$};
\node (ST') at (6.5,-0.5) {$TA$};
\node (TS0') at (9.3,-0.5) {$TB$.};
\path[->] 
(ST') edge node[scale=.7] [below] {$\alpha^\dagger$} (TS0')
(T) edge node[scale=.7] [left] {$ST{s_0}A$} (TS0X)
	edge node[scale=.7] [above] {$S\alpha^\dagger$} (T')
(TS0X) edge node[scale=.7] [left] {$\widehat{{m_0}A}$} (ST')
	edge[dotted] node[scale=.7] [above] {$S(dB\cdot S\alpha)^\dagger$} (TS0Y)
(T') edge node[scale=.7] [right] {$ST{s_0}B$} (TS0Y)
(TS0Y) edge node[scale=.7] [right] {$\widehat{{m_0}B}$} (TS0');
\end{tikzpicture}
\end{gathered}
\end{equation}
We will proceed proving that both squares ($\ast$) and ($\ast\ast$) in diagram (\ref{eq:lift-diag}) are commutative. Diagram ($\ast$) is the image through $S$ of a diagram ($\ast$')
\begin{center}
\begin{tikzpicture}
\node at (1.5,1) {($\ast$')};
\node (T) at (0,2) {$TA$};
\node (T') at (3,2) {$TB$};
\node (TS0X) at (0,0) {$TS_0A$};
\node (TS0Y) at (3,0) {$TS_0B$,};
\path[->] 
(T) edge node[scale=.7] [left] {$T{s_0}A$} (TS0X)
	edge node[scale=.7] [above] {$\alpha^\dagger$} (T')
(TS0X) edge node[scale=.7] [below] {$(dB\cdot S\alpha)^\dagger$} (TS0Y)
(T') edge node[scale=.7] [right] {$T{s_0}B$} (TS0Y);
\end{tikzpicture}
\end{center}
so it suffices to prove the commutativity of ($\ast$'). By left naturality of $(-)^\dagger$ and the equality $(t{S_0A})^\dagger=t^\dagger{S_0A}=1_{S_0A}$
\begin{center}
$T{s_0}A=(t{S_0A}\cdot I{s_0}A)^\dagger$,
\end{center} 
and therefore
\begin{align*}
(dB\cdot S\alpha)^\dagger \cdot T{s_0}A
&  = ((dB\cdot S\alpha)^\dagger \cdot t
{S_0A}\cdot I{s_0}A)^\dagger
& (\textrm{by associativity of}\;(-)^\dagger) 
\\
& = (dB\cdot S\alpha\cdot I{s_0}A)^\dagger
& (\textrm{by left unit of}\;(-)^\dagger) 
\\
& = (dB\cdot s{TB}\cdot\alpha)^\dagger
& (\textrm{by naturality of}\;s\;\textrm{and}\;Is_0=sI) 
\\
& = (T{s_0}B\cdot\alpha)^\dagger
& (\textrm{by (D4)}) 
\\
& = T{s_0}B\cdot\alpha^\dagger
& (\textrm{by associativity of}\;(-)^\dagger). 
\end{align*}
For ($\ast\ast$) we need to use axiom (\ref{lift-rel-alg:ext}) of Definition~\ref{defn:lift-rel-alg}. We can rewrite this axiom using Proposition~\ref{prop:lift-on-struct} and the definition of $d$, in the following way:
\begin{center}
\begin{tikzpicture}
\node (T) at (-0.2,2.5) {$IS_0M$};
\node (T') at (2.2,2.5) {$STN$};
\node (TS0Y) at (2.2,1) {$TS_0N$};
\node (ST') at (-0.2,-0.5) {$IM$};
\node (TS0') at (2.2,-0.5) {$TN$};
\path[->] 
(ST') edge node[scale=.7] [below] {$f$} (TS0')
(T) edge node[scale=.7] [left] {$Im$} (ST')
	edge node[scale=.7] [above] {$Sf$} (T')
(T') edge node[scale=.7] [right] {$d{N}$} (TS0Y)
(TS0Y) edge node[scale=.7] [right] {$Tn$} (TS0');

\node at (3.7,1) {$\Rightarrow$};

\node (T) at (5.2,2.5) {$STM$};
\node (T') at (7.6,2.5) {$STN$};
\node (TS0X) at (5.2,1) {$TS_0M$};
\node (TS0Y) at (7.6,1) {$TS_0N$};
\node (ST') at (5.2,-0.5) {$TM$};
\node (TS0') at (7.6,-0.5) {$TN$};
\path[->] 
(ST') edge node[scale=.7] [below] {$f^\dagger$} (TS0')
(T) edge node[scale=.7] [left] {$d{M}$} (TS0X)
	edge node[scale=.7] [above] {$Sf^\dagger$} (T')
(TS0X) edge node[scale=.7] [left] {$Tm$} (ST')
(T') edge node[scale=.7] [right] {$dN$} (TS0Y)
(TS0Y) edge node[scale=.7] [right] {$Tn$} (TS0');
\end{tikzpicture}
\end{center}
for $(M,\,m)$ and $(N,\,n)$ two $S_0$-algebras and $f$ a 2-cell in $\catK$. If we consider the case with $(M,\,m):=(S_0A,\,{m_0}A)$, $(N,\,n):=(S_0B,\,{m_0}B)$ and $f:=dB\cdot S\alpha$ we would get ($\ast\ast$) on the right (using again Proposition~\ref{prop:lift-on-struct}). So it is enough to prove that with these choices the diagram on the left is commutative, i.e.\
\begin{center}
\begin{tikzpicture}
\node (a) at (0,3) {$S^2IA$};
\node (a') at (-2,3) {$IS_0^2A$};
\node (b) at (3,3) {$S^2TB$};
\node (c) at (6,3) {$STS_0B$};
\node (d) at (6,1.5) {$TS_0^2B$};
\node (e) at (0,0) {$SIA$};
\node (e') at (-2,0) {$IS_0A$};
\node (f) at (3,0) {$STB$};
\node (g) at (6,0) {$TS_0B$.};

\path[->]
(a) edge node[scale=.7] [above] {$S^2\alpha$} (b)
	edge node[scale=.7] [left] {$m{IA}$} (e)
(a')	edge node[scale=.7] [left] {$I{m_0}A$} (e')
(b) edge node[scale=.7] [above] {$SdB$} (c)
	edge node[scale=.7] [left] {$m{TB}$} (f)
(c) edge node[scale=.7] [right] {$d{S_0B}$} (d)
(d) edge node[scale=.7] [right] {$T{m_0}B$} (g)
(e) edge node[scale=.7] [below] {$S\alpha$} (f)
(f) edge node[scale=.7] [below] {$dB$} (g);

\draw [double equal sign distance] (a) to (a');
\draw [double equal sign distance] (e) to (e');
\end{tikzpicture}
\end{center}
The square on the left commutes because $\compS$ is compatible with $I$, the one in the center commutes by naturality of $m$ and the diagram on the right is (D1) applied to~$B$. 
\end{proof}

In summary, Theorem~\ref{thm:lift-to-distr} and Lemma~\ref{lemma:distr-to-lift} give us two constructions:
\begin{center}
\begin{tikzpicture}
\node (a) at (0,0) {Rel. Distr. Laws};
\node (b) at (4,0) {Lift. to Alg.};

\path[->]
(a) edge[bend left] node[scale=.7] [above] {$(-)^\flat$} (b) 
(b) edge[bend left] node[scale=.7] [below] {$(-)^\sharp$} (a);
\end{tikzpicture}
\end{center} 
The following Theorem shows that these constructions are inverses of each other. 

\begin{theorem}
\label{thm:rel-distr-eq-lift}
Let $(X,\,I,\,T)$ be a relative monad in $\catK$ and $(S,\,S_0)$ a monad compatible with $I$. Relative distributive laws $d\colon ST\rightarrow TS_0$ of $T$ over $(S,\,S_0)$ are equivalent to liftings of $T$ to the algebras of $(S,\,S_0)$. 
\end{theorem}

\begin{proof}
Let start proving that $(d^\flat)^\sharp=d$:
\begin{align*}
(d^\flat)^\sharp
&  = d^\flat({m_0})\cdot ST{s_0}
& (\textrm{by definition of}\;(-)^\sharp) 
\\
& = T{m_0}\cdot d{S_0}\cdot ST{s_0}
& (\textrm{by definition of}\;(-)^\flat)
\\
& = T{m_0}\cdot TS_0{s_0} \cdot d
& (\textrm{by naturality of}\;d) 
\\
& = d
& (\textrm{by right unit of}\;S_0).
\end{align*}
On the other hand, using Proposition~\ref{prop:lift-on-struct} we can see that $(\widehat{T}^\sharp)^\flat=\widehat{T}$. 
\end{proof}

This result gives us the first equivalence of the counterpart of Beck's Theorem for relative distributive laws. We will prove the second equivalence in Section~\ref{sec:rrm-kl-ext} (Theorem~\ref{thm:rel-distr-eq-kl}), thus getting the entire counterpart (Theorem~\ref{thm:rel-beck}).

\section{Relative Right Modules and Kleisli Objects}
\label{sec:rrm-kl-ext}

In the formal theory of monads \cite{LackS:fortm, StreetR:fortm}, if we consider left modules (algebras) for a monad in $\catK^\op$ we get what are called \emph{right modules} for a monad in $\catK$. Using this duality, all the results for algebras can be translated easily to right modules. Unfortunately, when we consider relative monads it is not possible to take advantage of this duality. The issue is that the objects of $\Rel(\catK^\op)$ are not relative monads. Indeed, we get two 1-cells, a unit 2-cell together with an extension operator in $\catK^\op$, which is not the same as an operator in $\catK$. For this reason, we will need to define relative right modules explicitly. 

In this section we will study right modules for relative monads, which we will call \emph{relative right modules}. We will start giving the definition of the category~$\RRM_T(K)$ of $K$-indexed relative right $T$-modules for an indexing object $K\in\catK$ and a relative monad $(X,\,I,\,T)$ in $\catK$. This construction has to satisfy a couple of conditions. First of all, we want that, whenever $T$ is an actual monad in $\catK$, then the notion of relative right module and the usual one of right module should coincide. Moreover, when we consider $\catK=\Cat$, then the Kleisli category for a relative monad defined in \cite{AltChap:mnd-no-end} should represent the 2-functor $\RRM_T(-)$ of relative right modules.

Once provided the appropriate setting, we will use $\RRM_T(-)$ to construct a 2-category $\RLift(\catK)$ with objects relative monads in $\catK$, 1-cells \emph{lifting to relative right modules} and 2-cells \emph{maps of liftings to relative right modules}. These concepts generalise the ones of lifting to right modules in the monad case. Then, we show that $\RLift(\catK)$ is 2-isomorphic to $\Rel(\catK)$. Finally, thanks to this equivalence, we prove a Beck-like theorem stating that relative distributive laws are the same as liftings to relative right modules, with the appropriate definition of the latter. 

\subsection*{The Category of Relative Right Modules}
\addcontentsline{toc}{subsection}{The Category of Relative Right Modules}
\markboth{The Category of Relative Right Modules}{The Category of Relative Right Modules}

From now on we will consider a fixed relative monad $(X,\,I,\,T)\in\Rel(\catK)$ where $I\colon X_0\to X$. 

\begin{defn}
Let $K\in\catK$. A \emph{$K$-indexed relative right $T$-module} consists of a 1-cell $M\colon X_0\rightarrow K$ together with an operator $(-)_m\colon [I,\,T]\rightarrow[M,\,M]$ satisfying the following axioms:
\begin{itemize}
\item \emph{unit law}, i.e.\ $t_m=1_M$;
\item \emph{associativity}, i.e.\ for any 2-cells $h\colon IA\rightarrow TB$ and $k\colon IB\rightarrow TC$ (given any three 1-cells $A,B,C\colon O\rightarrow X_0$) 
\begin{center}
\begin{tikzpicture}
%
\node (IX') at (6.5,2) {$MA$};
\node (SX') at (8.5,2) {$MB$};
\node (SY') at (8.5,0) {$MC$.};
\draw[->] (IX') to node[scale=.7] [above]{$h_m$} (SX');
\draw[->] (SX') to node[scale=.7] [right]{$k_m$} (SY');
\draw[->] (IX') to [bend right]  node[scale=.7] (G'f) [left, xshift=-0.1cm] {$(k^\dagger\cdot h)_m$} (SY');
\end{tikzpicture}
\end{center}
\end{itemize}
\end{defn} 

\begin{defn}
Let $(M,\,(-)_m)$ and $(N,\,(-)_n)$ be two $K$-indexed relative right $T$-modules. A \emph{$K$-indexed relative right module morphism} $f\colon (M,\,(-)_m)\rightarrow (N,\,(-)_n)$ consists of a 2-cell $f\colon M\rightarrow N$ such that for any 2-cell $h\colon IA\rightarrow TB$ (given any pair of 1-cells $A,B\colon O\rightarrow X_0$):
\begin{center}
\begin{tikzpicture}
\node (IX) at (0,2) {$MA$};
\node (NA) at (0,0) {$NA$};
\node (SX) at (2,2) {$MB$};
\node (SY) at (2,0) {$NB$.};
\draw[->] (IX) to node[scale=.7] [above]{$h_m$} (SX);
\draw[->] (SX) to node[scale=.7] [right]{$fB$} (SY);
\draw[->] (IX) to  node[scale=.7] (G'f) [left] {$fA$} (NA);
\draw[->] (NA) to  node[scale=.7] (G'f) [below] {$h_n$} (SY);
\end{tikzpicture}
\end{center}
\end{defn}

Relative right modules and their morphisms in $\catK=\Cat$ become exactly \emph{modules over a relative monad and their morphisms} in the sense of \cite[Definition~9 and~14]{ahrens_2016}.

Clearly the definitions above form a category $\RRM_T(K)$ of $K$-indexed relative right modules. Therefore we have an induced 2-functor 
\begin{center}
$\RRM_T(-)\colon \catK\rightarrow\Cat$. 
\end{center}

\begin{rmk} 
We briefly show that if $I=1_X$, then $\RRM_T(K)$ is equal to the category of $K$-indexed right modules for a monad in the usual sense. \yellow Given \black a relative right module for a monad (i.e.\ a relative monad with $I=1_X$), we can prove that $\rho^M:=(1_T)_m$ is a $K$-indexed right module structure. We need to prove the following axioms
\begin{center}
\begin{tikzpicture}
\node (IX) at (0,2) {$MT^2$};
\node (MT) at (0,0) {$MT$};
\node (SX) at (2,2) {$MT$};
\node (SY) at (2,0) {$M$};
\draw[->] (IX) to node[scale=.7] [above]{$\rho^MT$} (SX);
\draw[->] (SX) to node[scale=.7] [right]{$\rho^M$} (SY);
\draw[->] (IX) to  node[scale=.7] (G'f) [left] {$Mn$} (MT);
\draw[->] (MT) to node[scale=.7] [below] {$\rho^M$} (SY);

\node (IX') at (6.5,2) {$M$};
\node (SX') at (8.5,2) {$MT$};
\node (SY') at (8.5,0) {$M$,};
\draw[->] (IX') to node[scale=.7] [above]{$Mt$} (SX');
\draw[->] (SX') to node[scale=.7] [right]{$\rho^M$} (SY');
\draw[double equal sign distance] (IX') to [bend right]  node[scale=.7] (G'f) [left, xshift=-0.1cm] {} (SY');
\end{tikzpicture}
\end{center}
where $n:=(1_T)^\dagger$. We can deduce them in the following way:
\begin{align*}
\rho^M\cdot Mn
&  = (1_T)_m\cdot M(1_T)^\dagger
& (\textrm{by definition}) 
\\
& = (1_T\cdot (1_T)^\dagger)_m=((1_T)^\dagger\cdot1_{T^2})_m
& (\textrm{by left naturality of}\;(-)_m)
\\
& = (1_T)_m\cdot(1_{T^2})_m
& (\textrm{by associativity of}\;(-)_m) 
\\
& = (1_T)_m\cdot(1_{T})_mT
& (\textrm{by left naturality of}\;(-)_m)
\\
&  = \rho^M\cdot \rho^MT
& (\textrm{by definitions}) 
\\
& & \\
\rho^M\cdot Mt
&  = (1_T)_m\cdot Mt
& (\textrm{by definitions}) 
\\
& = (1_T\cdot t)_m=t_m
& (\textrm{by left naturality of}\;(-)_m)
\\
& = 1_M
& (\textrm{by unit law for}\;(-)_m). 
\end{align*}
On the other side, if we begin with a right module structure $\rho^M\colon MT\rightarrow M$, we can define a relative right module structure as 
$$f\colon A\rightarrow TB\longmapsto f_m:=\rho^MB\cdot Mf$$
Then by definition $t_m=\rho^M\cdot Mt$ which is equal to $1_M$ by the unit axiom for~$\rho^M$. Therefore we have left to prove that, for any $f\colon A\rightarrow TB$ and $g\colon B\rightarrow TC$, $g_m\cdot f_m=(g^\dagger\cdot f)_m$ where $g^\dagger=nC\cdot Tg$. We can prove this by looking at the diagram 
\begin{center}
\begin{tikzpicture}
\node (S2TS0) at (3,4) {$MA$};
\node (STS0u) at (6,4) {$MTB$};
\node (STS02) at (9,4) {$MB$};
\node (MT2C) at (6,2) {$MT^2C$};
\node (TS02) at (9,2) {$MTC$};
\node (STS0d) at (3,0) {$MTB$};
\node (MTC) at (6,0) {$MTC$};
\node (TS0) at (9,0) {$MC$.};

\path[->] 
(S2TS0) edge node[scale=.7] [above] {$Mf$} (STS0u)
	edge node[scale=.7] [left] {$Mf$} (STS0d)
(STS0u) edge node[scale=.7] [right] {$MTg$} (MT2C)
	edge node[scale=.7] [above] {$\rho^MB$} (STS02)
(STS02) edge node[scale=.7] [right] {$Mg$} (TS02)
(TS02) edge node[scale=.7] [right] {$\rho^MC$} (TS0)
(STS0d) edge node[scale=.7] [below] {$Mg^\dagger$} (MTC)
(MTC) edge node[scale=.7] [below] {$\rho^MC$} (TS0)
(MT2C) edge node[scale=.7] [below] {$\rho^MC$} (TS02)
	edge node[scale=.7] [left] {$MnC$} (MTC)
;
\draw[double equal sign distance] (STS0u) to  node[scale=.7] (G'f) [left, xshift=-0.1cm] {} (STS0d);
\end{tikzpicture}
\end{center}
The commutativity follows from the naturality and multiplication axiom for $\rho^M$. Using indexing and left naturality for $(-)_m$ we can see that these two constructions provide a bijection between relative right modules for a monad and the usual notion of right modules. 
\end{rmk}

\begin{defn}
We say that a relative monad $(X,\,I,\,T)\in\Rel(\catK)$ has a \emph{relative Kleisli object} if $\RRM_T(-)$ is representable. We will denote a representing object with $X_{I,T}$. 
\end{defn}

\begin{eg}
\label{eg:Kl-obj-cat}
Let $\Kl(T)$ denote the Kleisli category for a relative monad $(\catX,\,I,\,T)$ in $\Cat$ \cite{AltChap:mnd-no-end}. We can see that this gives us a relative Kleisli object for~$T$. Let $(M,\,(-)_m)$ be a $\catK$-indexed relative right module. We can define a functor $\bar{M}\colon \Kl(T)\rightarrow\catK$ the same as $M$ on objects, and for any map $f\colon x\nrightarrow y\in\Kl(T)$, i.e.\ $f\colon Ix\rightarrow Ty\in\catX$, $\bar{M}f:=f_m\colon Mx\rightarrow My$. The unit law and associativity of $(-)_m$ ensure that $\bar{M}$ respects identities and composition respectively. Moreover, $\bar{M}$ defined in this way is such that $\bar{M}J_0=M$. The equality on objects is trivially true, whilst for the action on maps we have 
\begin{align*}
\bar{M}J_0f
&  = (Tf\cdot t_x)_m
& (\textrm{by definitions}) 
\\
& = Mf\cdot(t_x)_m
& (\textrm{by right naturality of}\;(-)_m)
\\
& = Mf\cdot 1_M=Mf
& (\textrm{by unit law for}\;(-)_m) .
\end{align*} 
On the other hand, if we start with a functor $\bar{M}\colon \Kl(T)\rightarrow\catK$ then we can define a $\catK$-indexed relative right module in the following way. First of all we define $M\colon \catX_0\rightarrow\catK$ as $\bar{M}J_0$. Then as operator structure we define, for any map $f\colon Ix\rightarrow Ty$, $f_m$ as $\bar{M}f\colon Mx\rightarrow My$. In an analogous way as before, the functoriality of $\bar{M}$ proves the unit law and associativity for $(-)_m$ defined in this way. 
\end{eg}

In the formal theory of monads \cite{StreetR:fortm}, the category of $K$-indexed right $T$-modules is shown to be equivalent to $\Mnd(\catK^\op)^\op[\,(X,\,T),\,(K,\,1_K)\,]$, and therefore the Kleisli object construction gives a left adjoint of the inclusion $\catK\rightarrow\Mnd(\catK^\op)^\op$ sending an object $K\in\catK$ to the unital monad $(K,\,1_K)$. The following proposition shows why this is not possible in our setting. 

\begin{prop}
Let $(X,\,I,\,T)$ be a relative monad in $\catK$ and $K\in\catK$ an object. Then the category of relative monad morphisms $\Rel(\catK)[\,(X,\,I,\,T),\,(K,\,1_K,\,1_K)\,]$ is a subcategory of $\RRM_T(K)$. 
\end{prop}

\begin{proof}
A relative monad map $(M,\,M_0,\,\rho):(X,\,I,\,T)\rightarrow(K,\,1_K,\,1_K)$ consists of a pair of maps $M\colon X\rightarrow K$ and $M_0\colon X_0\rightarrow K$ such that $MI=M_0$, and a 2-cell $\rho\colon MT\rightarrow M_0$ satisfying the axioms (for any 2-cell $f\colon IA\rightarrow TB$)
\begin{center}
\begin{tikzpicture}
\node at (-1,1) {(i)};
\node (IX) at (0,2) {$MI$};
\node (SX) at (2,2) {$MT$};
\node (SY) at (2,0) {$M_0$};
\draw[->] (IX) to node[scale=.7] [above, xshift=-0.1cm]{$Mt$} (SX);
\draw[->] (SX) to node[scale=.7] [right]{$\rho$} (SY);
\draw[double equal sign distance] [bend right] (IX) to  node[scale=.7] (G'f) [left, yshift=-0.2cm, xshift=0.1cm] {} (SY);
\node at (4.5,1) {(ii)};
\node (a) at (6,3) {$MSA$};
\node (b) at (8.5,3) {$M_0A$};
\node (FSB) at (8.5,1) {$MSB$};
\node (c) at (6,-1) {$MSB$};
\node (d) at (8.5,-1) {$M_0B$.};
\draw[->] (a) to node[scale=.7] [above]{$\rho A$} (b);
\draw[->] (c) to node[scale=.7] [below]{$\rho B$} (d);
\draw[->] (a) to node[scale=.7] [left]{$Mf^\dagger$} (c);
\draw[->] (b) to node[scale=.7] [right]{$Mf$} (FSB);
\draw[->] (FSB) to node[scale=.7] [right]{$\rho B$} (d);
\end{tikzpicture}
\end{center}
We can endow $M_0$ with a relative right module structure defining its operator, for any 2-cell $f\colon IA\rightarrow TB$, as $f_m:=\rho B\cdot Mf$. The unit law is guaranteed by diagram (i) above, while to prove associativity it is enough to precompose the two composite in diagram (ii) with $Mf$. 

Moreover, let $(\alpha,\,\alpha_0)\colon (M,\,M_0,\,\rho)\rightarrow(M',\,M'_0,\,\rho')$ be a relative monad transformation. Then the diagram
\begin{center}
\begin{tikzpicture}
\node (a0) at (-0.7,2) {$M_0A$};
\node (a) at (1,2) {$MIA$};
\node (b) at (3,2) {$MSB$};
\node (c) at (5,2) {$M_0B$};
\node (a'0) at (-0.7,0) {$M'_0A$};
\node (a') at (1,0) {$M'IA$};
\node (b') at (3,0) {$M'SB$};
\node (c') at (5,0) {$M'_0B$};

\path[->] 
(a) edge node[scale=.7] [above]{$Mf$} (b)
	edge node[scale=.7] [left]{$\alpha IA$} (a')
(a0) edge node[scale=.7] [left]{$\alpha_0A$} (a'0)
(b) edge node[scale=.7] [above]{$\rho B$} (c)
	edge node[scale=.7] [left]{$\alpha SB$} (b')
(c) edge node[scale=.7] [right]{$\alpha_0 B$} (c')
(a') edge node[scale=.7] [below]{$M'f$} (b')
(b') edge node[scale=.7] [below]{$\rho' B$} (c')
;

\draw[double equal sign distance] (a0) to  node[scale=.7] {} (a);
\draw[double equal sign distance] (a'0) to  node[scale=.7] {} (a');
\end{tikzpicture}
\end{center} 
is commutative using $\alpha_0=\alpha I$, the naturality of $\alpha$ and the axiom for a relative monad transformation. Therefore $\alpha_0$ is a relative right module morphism.
\end{proof}

An object of $\Rel(\catK)[\,(X,\,I,\,T),\,(K,\,1_K,\,1_K)\,]$ is similar to the usual notion of a right module, having a straightforward right action $\rho\colon MT\rightarrow M_0$ satisfying axioms similar to the one of a right module. So, one could wonder if this would be the appropriate definition for a relative right module over a relative monad. The main problem with this definition is that, in a general 2-category $\catK$, is not possible to find a morphism from $\catK(X,K)$ to $\Rel(\catK)[\,(X,\,I,\,T),\,(K,\,1_K,\,1_K)\,]$. We need a map of this kind to form a diagram like (\ref{diag:rrm}), which will be crucial in the next section, where we will define a lifting to relative right modules. Instead, for any $(X,\,I,\,T)\in\Rel(\catK)$ we can consider the diagram 
\begin{equation}
\label{diag:rrm}
\begin{gathered}
\begin{tikzpicture}
\node (a) at (0,0) {$\catK(X,-)$};
\node (b) at (2,1.5) {$\RRM_T(-)$};
\node (c) at (4,0) {$\catK(X_0,-)$};

\path[->] 
(a) edge node[scale=.7] [left,xshift=-0.1cm,yshift=0.1cm]{$U^\ast_X$} (b)
	edge node[scale=.7] (I) [below]{$-\circ I$} (c)
(b) edge node[scale=.7] [right, xshift=0.1cm,yshift=0.1cm]{$J^\ast_{X_0}$} (c)
;

\draw[-{Implies},double distance=1.5pt,shorten >=10pt,shorten <=10pt] (I) to node[scale=.7] [right, xshift=0.1cm] {$t$} (b);

\end{tikzpicture}
\end{gathered}
\end{equation}
where $J_{X_0}^\ast\colon \RRM_T(-)\rightarrow\catK(X_0,\,-)$ is a forgetful natural transformation, sending a $K$-indexed relative right $T$-module to its underlying 1-cell $M\colon X_0\rightarrow K$, and  $U^\ast_X\colon \catK(X,\,-)\rightarrow\RRM_T(-)$ is a natural transformation defined as, for any $K\in\catK$, 
\begin{center}
$(U^\ast_X)_K\colon \catK(X,\,K)\longrightarrow\RRM_T(K)$ \\ \vspace{0.1cm}
\hspace{1.4cm}$(M\colon X\rightarrow K)\;\longmapsto (MT,\,M(-)^\dagger)$.
\end{center} 

\begin{rmk} 
\label{rmk:rel-adj-vs-yon}
Whilst using relative algebras we were able to construct a relative adjunction (Lemma~\ref{lemma:rel-EM-to-rel-mnd}), the diagram in (\ref{diag:rrm}) does not always represent one. However, even if we would have such a relative adjunction, then it would still not seem possible to prove a similar result to Theorem~\ref{thm:EM-obj}. Indeed, to define a relative adjunction in a 2-category we use operators, which are not self dual. Thus the contravariant Yoneda embedding $Y\colon \catK^\op\rightarrow[\catK,\,\Cat]$ does not preserve operators and so relative adjunctions, while the covariant Yoneda embedding $Y\colon \catK\rightarrow[\catK^\op,\,\Cat]$ does. Nevertheless, there are some 2-categories where we find a relative adjunction using Kleisli objects (for instance $\Cat$, see \cite[Section~2.3]{AltChap:mnd-no-end}). 
\end{rmk}


\subsection*{The 2-category $\RLift(\catK)$}
\addcontentsline{toc}{subsection}{The 2-category $\RLift(\catK)$}
\markboth{The 2-category $\RLift(\catK)$}{The 2-category $\RLift(\catK)$}

Given a 2-category $\catK$ we want to use the notion of relative right modules to define a 2-category $\RLift(\catK)$ with the same objects as $\Rel(\catK)$ but 1- and 2-cells defined as \textit{lifting to relative right modules}. Let us start fixing some notation. From now on we will consider two relative monads $(X,\,I,\,S)$ and $(Y,\,J,\,T)$ in $\catK$. They induce the following diagrams: 

\begin{center}
\begin{tikzpicture}
\node (a) at (0,0) {$\catK(X,-)$};
\node (b) at (2,1.5) {$\RRM_S(-)$};
\node (c) at (4,0) {$\catK(X_0,-)$,};

\path[->] 
(a) edge node[scale=.7] [left,xshift=-0.1cm,yshift=0.1cm]{$U^\ast_X$} (b)
	edge node[scale=.7] (I) [below]{$-\circ I$} (c)
(b) edge node[scale=.7] [right, xshift=0.1cm,yshift=0.1cm]{$J^\ast_{X_0}$} (c)
;

\draw[-{Implies},double distance=1.5pt,shorten >=10pt,shorten <=10pt] (I) to node[scale=.7] [right, xshift=0.1cm] {$s$} (b);

\node (a) at (7,0) {$\catK(Y,-)$};
\node (b) at (9,1.5) {$\RRM_T(-)$};
\node (c) at (11,0) {$\catK(Y_0,-)$.};

\path[->] 
(a) edge node[scale=.7] [left,xshift=-0.1cm,yshift=0.1cm]{$U^\ast_Y$} (b)
	edge node[scale=.7] (J) [below]{$-\circ J$} (c)
(b) edge node[scale=.7] [right, xshift=0.1cm,yshift=0.1cm]{$J^\ast_{Y_0}$} (c)
;

\draw[-{Implies},double distance=1.5pt,shorten >=10pt,shorten <=10pt] (J) to node[scale=.7] [right, xshift=0.1cm] {$t$} (b);
\end{tikzpicture}
\end{center}

\begin{defn}
\label{defn:lift-to-rrm}
Let $(X,\,I,\,S)$ and $(Y,\,J,\,T)$ be two relative monads in $\catK$. A \emph{lifting to relative right modules} consists of two 1-cells $F\colon X\rightarrow Y$ and $F_0\colon X_0\rightarrow Y_0$, a natural transformation $\tilde{F}\colon \RRM_T(-)\rightarrow\RRM_S(-)$ and a modification $\tilde{\phi}$ of the form 
\begin{center}
\begin{tikzpicture}
\node (a) at (0,0) {$\RRM_T(-)$};
\node (b) at (3,2) {$\catK(X,-)$};
\node (c) at (3,0) {$\RRM_S(-)$};
\node (d) at (0,2) {$\catK(Y,-)$};

\path[->] 
(d) edge node[scale=.7] (F) [above]{$-\circ F$} (b)
	edge node[scale=.7] [left,xshift=-0.1cm,yshift=0.1cm]{$U^\ast_Y$} (a)
(a) edge node[scale=.7] (I) [below]{$\tilde{F}$} (c)
(b) edge node[scale=.7] [right, xshift=0.1cm,yshift=0.1cm]{$U^\ast_X$} (c)
;

\draw[-{Implies},double distance=1.5pt,shorten >=20pt,shorten <=20pt] (F) to node[scale=.7] [right, xshift=0.1cm] {$\tilde{\phi}$} (I);
\end{tikzpicture}
\end{center}
satisfying the axioms: 
\begin{enumerate}[(i)]
\item $FI=JF_0$ and the following diagram commutes
\begin{center}
\begin{tikzpicture}
\node (a) at (0,0) {$\catK(Y_0,-)$};
\node (b) at (3,2) {$\RRM_S(-)$};
\node (c) at (3,0) {$\catK(X_0,-)$.};
\node (d) at (0,2) {$\RRM_T(-)$};

\path[->] 
(d) edge node[scale=.7] (F) [above]{$\tilde{F}$} (b)
	edge node[scale=.7] [left,xshift=-0.1cm,yshift=0.1cm]{$J^\ast_{Y_0}$} (a)
(a) edge node[scale=.7] (I) [below]{$-\circ F_0$} (c)
(b) edge node[scale=.7] [right, xshift=0.1cm,yshift=0.1cm]{$J^\ast_{X_0}$} (c)
;
\end{tikzpicture}
\end{center}

\item \label{ax:lift-unit} The following pasting diagrams are equal \\
\begin{tikzpicture}
\node (a') at (0,2) {$\catK(Y,-)$};
\node (b') at (2,3.5) {$\RRM_T(-)$};
\node (c') at (3,0) {$\catK(Y_0,-)$};
\node (b) at (6,3.5) {$\RRM_S(-)$};
\node (c) at (7,0) {$\catK(X_0,-)$};
\node at (9,1.75) {$=$};

\path[->] 
(b) edge node[scale=.7] [right, xshift=0.1cm,yshift=0.1cm]{$J^\ast_{X_0}$} (c)

(a') edge node[scale=.7] [left,xshift=-0.1cm,yshift=0.1cm]{$U^\ast_Y$} (b')
	edge node[scale=.7] (J) [left, xshift=-0.1cm, yshift=-0.2cm]{$-\circ J$} (c')
(b') edge node[scale=.7] [right, xshift=0.1cm,yshift=0.1cm]{$J^\ast_{Y_0}$} (c')
	edge node[scale=.7] [above]{$\tilde{F}$} (b)
(c')edge node[scale=.7] [below]{$-\circ F_0$} (c)
;


\draw[-{Implies},double distance=1.5pt,shorten >=20pt,shorten <=20pt] (J) to node[scale=.7] [right, xshift=0.1cm, yshift=-0.2cm] {$t$} (b');

\end{tikzpicture} \\
\begin{tikzpicture}
\node at (-3,0) {};
\node at (0,1.75) {$=$};
\node (a') at (2,2) {$\catK(Y,-)$};
\node (b') at (4,3.5) {$\RRM_T(-)$};
\node (c') at (5,0) {$\catK(Y_0,-)$};
\node (a) at (6,2) {$\catK(X,-)$};
\node (b) at (8,3.5) {$\RRM_S(-)$};
\node (c) at (9,0) {$\catK(X_0,-)$.};

\path[->] 
(a) edge node[scale=.7] [left,xshift=-0.1cm,yshift=0.1cm]{$U^\ast_X$} (b)
	edge node[scale=.7] (I) [left, xshift=-0.1cm, yshift=-0.2cm]{$-\circ I$} (c)
(b) edge node[scale=.7] [right, xshift=0.1cm,yshift=0.1cm]{$J^\ast_{X_0}$} (c)

(a') edge node[scale=.7] [left,xshift=-0.1cm,yshift=0.1cm]{$U^\ast_Y$} (b')
	edge node[scale=.7] (J) [left, xshift=-0.1cm, yshift=-0.2cm]{$-\circ J$} (c')
(b') edge node[scale=.7] (Ftilde) [above]{$\tilde{F}$} (b)
(c')edge node[scale=.7] [below]{$-\circ F_0$} (c)
(a')edge node[scale=.7] (F) [below]{$-\circ F$} (a)
;

\draw[-{Implies},double distance=1.5pt,shorten >=20pt,shorten <=20pt] (F) to node[scale=.7] [right, xshift=0.4cm, yshift=-0.1cm] {$\tilde{\phi}$} (Ftilde);

\draw[-{Implies},double distance=1.5pt,shorten >=20pt,shorten <=20pt] (I) to node[scale=.7] [right, xshift=0.2cm, yshift=-0.1cm] {$s$} (b);


\end{tikzpicture}

\item \label{ax:lift-kl-ext} Let us denote with $(-)_{\tilde{F}m}$ and $(-)_{\tilde{F}T}$ the relative right module structure operators of $\tilde{F}(M,\,(-)_m)$ and $\tilde{F}(T,\,(-)^\dagger)$ respectively. Then, for any 2-cell $f\colon IA\rightarrow TB$ and any $K$-indexed relative right module $(M,\,(-)_m)$, the action of $\tilde{F}$ has to be
\begin{center}
$f_{\tilde{F}m}=(f_{\tilde{F}T}\cdot tF_0A)_m$.
\end{center}
\end{enumerate}
\end{defn}

Since the last part of the definition above might seem a bit ad hoc, let us briefly explain where it comes from. The idea is that we want to write the action of $\tilde{F}$ on relative right module in terms of the \emph{free} relative right modules. In the case of relative algebras, this property followed from the other axioms (\Cref{prop:lift-on-struct}). This will make sure that in the definition a lifting of a monad to relative right modules (\Cref{defn:ext-to-kl}) all of the structure of the monad is lifted. 

The following Lemma gives us a nice way to describe the action of $\tilde{F}$ in terms of a particular 2-cell $\phi$. Thanks to this we can prove that a lifting to relative right modules is equivalent to a morphism of relative monads, as shown in \Cref{prop:lift-iff-rel-1-cells}.

\begin{lemma}
\label{lemma:act-of-lift}
Let $(F,\,F_0,\,\tilde{F},\,\tilde{\phi})$ be a lifting to relative right modules. Let us denote by $\phi\colon  FS\rightarrow TF_0$ the component of $J^\ast_{X_0}\tilde{\phi}$ relative to $1_Y$.
\begin{center}
\begin{tikzpicture}
\node (a) at (0,0) {$\RRM_T(Y)$};
\node (b) at (3,2) {$\catK(X,Y)$};
\node (c) at (3,0) {$\RRM_S(Y)$};
\node (d) at (0,2) {$\catK(Y,Y)$};

\path[->] 
(d) edge node[scale=.7] (F) [above]{$-\circ F$} (b)
	edge node[scale=.7] [left,xshift=-0.1cm,yshift=0.1cm]{$U^\ast_Y$} (a)
(a) edge node[scale=.7] (I) [below]{$\tilde{F}$} (c)
(b) edge node[scale=.7] [right, xshift=0.1cm,yshift=0.1cm]{$U^\ast_X$} (c)
;

\draw[-{Implies},double distance=1.5pt,shorten >=20pt,shorten <=20pt] (F) to node[scale=.7] [right, xshift=0.1cm] {$\tilde{\phi}$} (I);
\node (a') at (0,-2) {$\catK(Y_0,Y)$};
\node (c') at (3,-2) {$\catK(X_0,Y)$};

\path[->] 
(a) edge node[scale=.7] [left,xshift=-0.1cm,yshift=0.1cm]{$J^\ast_{Y_0}$} (a')
(a') edge node[scale=.7] (I) [below]{$-\circ F_0$} (c')
(c) edge node[scale=.7] [right, xshift=0.1cm,yshift=0.1cm]{$J^\ast_{X_0}$} (c')
;

\node (1) at (5.5,2) {$1_Y$};
\node (2) at (8,2) {$F$};
\node (3) at (8,-1) {$FS$};
\draw[|->] (1) edge node {} (2);
\draw[|->] (2) edge node {} (3);

\node (2') at (5.5,-2) {$T$};
\node (3') at (7,-2) {$TF_0$};
\draw[|->] (1) edge node {} (2');
\draw[|->] (2') edge node {} (3');

\draw[-{Implies},double distance=1.5pt] (3) to node[scale=.7] [right, yshift=-0.2cm, xshift=+0.1cm] {$\phi$} (3');

\end{tikzpicture}
\end{center}
Then, for any 2-cell $f\colon IA\rightarrow SB$, 
\begin{center}
$f_{\tilde{F}m}=(\phi B\cdot Ff)_m$.
\end{center}
\end{lemma}

\begin{proof}
We know that $\phi\colon FS\rightarrow TF_0$ is a map in $\RRM_S(X)$, where the structure operators of $FS$ and $TF_0$ are respectively $F(-)^\dagger_S$ and $(-)_{\tilde{F}T}$. Therefore, the following diagram is commutative
\begin{center}
\begin{tikzpicture}
\node (S2T) at (-0.5,4) {$FIA$};
\node (s2t') at (2,4) {$JF_0A$};
\node (STS0) at (-0.5,2) {$FSA$};
\node (TS02) at (-0.5,0) {$FSB$};
\node (ST) at (2,2) {$TF_0A$};
\node (TS0) at (2,0) {$TF_0B$.};
\draw[double equal sign distance]  (S2T) to  node[scale=.7] (G'f) [left, yshift=-0.2cm, xshift=0.1cm] {} (s2t');
\path[->] 
(S2T) edge node[scale=.7] [left] {$FsA$} (STS0)
(ST) edge node[scale=.7] [right] {$f_{\tilde{F}T}$} (TS0)
(STS0) edge node[scale=.7] [left] {$Ff^\dagger_S$} (TS02)
	edge node[scale=.7] [above] {$\phi A$} (ST)
(TS02) edge node[scale=.7] [below] {$\phi B$} (TS0);
\draw[->] (s2t') edge node[scale=.7]  [right] {$tF_0A$} (ST);
\end{tikzpicture}
\end{center}
Indeed, the top square commutes by part (\ref{ax:lift-unit}) of Definition~\ref{defn:lift-to-rrm} and the bottom square because $\phi$ is a relative right module map. Thus, we can deduce
\begin{align*}
f_{\tilde{F}m}
&  = (f_{\tilde{F}T}\cdot tF_0A)_m
& (\textrm{by part (\ref{ax:lift-kl-ext}) of Definition\til\ref{defn:lift-to-rrm}}) &
\\
& = (\phi B\cdot Ff^\dagger_S\cdot FsA)_m
& (\textrm{diagram above}) &
\\
& = (\phi B\cdot Ff)_m
& (\textrm{by left unit law of}\;S). & \qedhere
\end{align*}  \end{proof} 

\begin{prop}
\label{prop:lift-iff-rel-1-cells}
Let $(X,\,I,\,S)$ and $(Y,\,J,\,T)$ be two relative monads in $\catK$. Then a lifting to relative right modules is equivalent to a relative monad morphism between them. 
\end{prop}

\begin{proof}
We will start proving that given a lifting to relative right modules we get a relative monad morphism. Let us denote with $\phi\colon  FS\rightarrow TF$ the 2-cell given in Lemma~\ref{lemma:act-of-lift}. Axiom (\ref{ax:lift-unit}) of Definition~\ref{defn:lift-to-rrm} is equivalent to the unit law for $(F,\,F_0,\,\phi)$ seen as a relative monad morphism. Moreover, the extension law for it follows from the fact that the components of $\tilde{\phi}$ are relative right modules morphisms and axioms (\ref{ax:lift-unit}) and (\ref{ax:lift-kl-ext}) of Definition~\ref{defn:lift-to-rrm}. More precisely, we have to prove that for any 2-cell $f\colon IA\rightarrow SB$ the following diagram is commutative
\begin{center}
\begin{tikzpicture}
\node (T) at (0,2) {$FSA$};
\node (T') at (0,0) {$FSB$};
\node (ST') at (2.2,2) {$TF_0A$};
\node (TS0') at (2.2,0) {$TF_0B$.};
\path[->] 
(ST') edge node[scale=.7] [right] {$(\phi B\cdot Sf)^\dagger_T$} (TS0')
(T) edge node[scale=.7] [above] {$\phi A$} (ST')
	edge node[scale=.7] [left] {$Ff^\dagger_S$} (T')
(T') edge node[scale=.7] [below] {$\phi B$} (TS0');
\end{tikzpicture}
\end{center}
We know that $\phi\colon FS\rightarrow TF_0$ is a map in $\RRM_S(X)$, therefore for any $f\colon IA\rightarrow SB$ we have, using that the structure operators of $FS$ and $TF_0$ are respectively $Ff^\dagger_S$ and $f_{\tilde{F}T}$,
\begin{center}
\begin{tikzpicture}
\node (T) at (0,2) {$FSA$};
\node (T') at (0,0) {$FSB$};
\node (ST') at (2.2,2) {$TF_0A$};
\node (TS0') at (2.2,0) {$TF_0B$.};
\path[->] 
(ST') edge node[scale=.7] [right] {$f_{\tilde{F}T}$} (TS0')
(T) edge node[scale=.7] [above] {$\phi A$} (ST')
	edge node[scale=.7] [left] {$Ff^\dagger_S$} (T')
(T') edge node[scale=.7] [below] {$\phi B$} (TS0');
\end{tikzpicture}
\end{center}
Hence, we just need to prove that $f_{\tilde{F}T}=(\phi B\cdot Ff)^\dagger_T$, which is just a particular instance of Lemma~\ref{lemma:act-of-lift}. 

On the other hand if we start with a relative monad morphism $(F,\,F_0,\,\phi)$, we can define $\tilde{F}\colon \RRM_T(-)\rightarrow\RRM_S(-)$, for any $K\in\catK$ and $(M,\,(-)_m)\in\RRM_T(K)$, as
\begin{center}
$\tilde{F}(M,\,(-)_m):=(MF_0,\,(\phi B\cdot F-)_m)$.
\end{center}
First of all we need to prove that $(\phi B\cdot F-)_m$ is a relative right module operator. 
\begin{align*}
\textrm{Unit Law:} & & \\
(\phi B\cdot Fs)_m
&  = (tF_0)_m
& (\textrm{by part (\ref{ax:lift-unit}) of Definition\til\ref{defn:lift-to-rrm}}) 
\\
& = 1_{MF_0}
& (\textrm{by unit law for}\;(-)_m).
\\
\textrm{Associativity:} & & \\
(k^\dagger\cdot h)_{\tilde{S}m}
&  = (\,\phi C\cdot Fk^\dagger_S\cdot Fh)\,)_m
& (\textrm{by definition}) 
\\
& = (\,(\phi C\cdot Fk)^\dagger_T\cdot \phi B\cdot Fh)\,)_m
& (\textrm{by Kleisli ext law for}\;\phi)
\\
& = (\phi C\cdot Fk)_m\cdot (\phi B\cdot Fh)_m
& (\textrm{by associativity of}\;(-)_m)
\\
& = k_{\tilde{F}m}\cdot h_{\tilde{F}m}
& (\textrm{by definition}).
\end{align*} 
Moreover if $g\colon (M,\,(-)_m)\rightarrow(N,\,(-)_n)$ is a map in $\RRM_T(K)$, then applying the axiom for $g$ to $\phi B\cdot Ff$ we get that $gF_0\colon \tilde{F}(M,\,(-)_m)\rightarrow\tilde{F}(N,\,(-)_n)$ is in $\RRM_S(K)$. Therefore $\tilde{F}\colon \RRM_T(-)\rightarrow\RRM_T(-)$ is well defined. 

Then, we can define the component of $\tilde{\phi}$ at $K\in\catK$ and $(M,\,(-)_m)\in\RRM_T(K)$ as~$M\phi$. Looking at the definition of $\tilde{F}$ on relative right actions, we can see that the axiom for $M\phi$ to be a relative right module morphism is the same as the extension axiom for $(F,\,F_0,\,\phi)$. 

Now we need to prove all the axioms of a lifting to relative right modules. The first one follows from definition, and part (\ref{ax:lift-unit}) is equivalent to the unit one for a relative monad morphism. Finally, we can easily check that part (\ref{ax:lift-kl-ext}) of Definition~\ref{defn:lift-to-rrm} is satisfied, as $f_{\tilde{F}m}=(\phi B\cdot Ff)_m$ and $f_{\tilde{F}T}=(\phi B\cdot Sf)^\dagger$ by definition, and so
\begin{center}
$f_{\tilde{F}m}=(\phi B\cdot Ff)_m=(\,(\phi B\cdot Ff)^\dagger\cdot tF_0A\,)_m=(f_{\tilde{F}T}\cdot tF_0A)_m$.
\end{center}
Lemma~\ref{lemma:act-of-lift} guarantees that these constructions are inverses of each other. 
\end{proof}

\begin{defn}
\label{defn:map-lift-rrm}
Let $(F,\,F_0,\,\tilde{F},\,\tilde{\phi})$ and $(F',\,F'_0,\,\tilde{F'},\,\tilde{\phi'})$ be two liftings to relative right modules from $(X,\,I,\,S)$ to $(Y,\,J,\,T)$, two relative monads in $\catK$. A \emph{map of liftings to relative right modules} $(p,\,p_0,\,\tilde{p})\colon (F,\,F_0,\,\tilde{F},\,\tilde{\phi})\rightarrow(F',\,F'_0,\,\tilde{F'},\,\tilde{\phi'})$ consists of two 2-cells\til$p\colon F\rightarrow F'$ and $p_0\colon F_0\rightarrow F'_0$ and a modification\til$\tilde{p}\colon \tilde{F}\rightarrow\tilde{F'}$ such that:
\begin{enumerate}[(i)]
\item $Jp_0=pI$ and the following pasting diagrams are equal
\begin{center}
\begin{tikzpicture}
\node (QGF) at (0,6) {$\RRM_T(-)$}; 
\node (QG'F) at (4,6) {$\RRM_S(-)$};
\node (GTF) at (0,3) {$\catK(Y_0,-)$};
\node (G'TF) at (4,3) {$\catK(X_0,-)$};

\path[->]
(QGF) edge [bend left=20] node[scale=.7] (Qq'F) [above] {$\tilde{F}$} (QG'F)
	edge node[scale=.7] [left] {$J^\ast_{Y_0}$} (GTF)
	edge [bend right=20] node[scale=.7] (QqF) [below] {$\tilde{F'}$} (QG'F)
(QG'F) edge node[scale=.7] [right] {$J^\ast_{X_0}$} (G'TF)
(GTF) edge [bend right=20] node[scale=.7] (q'TF) [below] {$\tilde{F'}$} (G'TF);

\draw[-{Implies},double distance=1.5pt,shorten >=10pt,shorten <=10pt] (Qq'F) to node[scale=.7] [right,xshift=0.2cm] {$\tilde{p}$} (QqF);

\node at (5.5,4.5) {$=$};

\node (GTFR) at (7,6) {$\RRM_T(-)$};
\node (G'TFR) at (11,6) {$\RRM_S(-)$};
\node (GFSR) at (7,3) {$\catK(Y_0,-)$};
\node (G'FSR) at (11,3) {$\catK(X_0,-)$;};

\path[->]
(GTFR) edge [bend left=20] node[scale=.7] (q'TFR) [above] {$\tilde{F}$} (G'TFR)
	edge node[scale=.7][left] {$J^\ast_{Y_0}$} (GFSR)
(G'TFR) edge node[scale=.7] [right] {$J^\ast_{X_0}$} (G'FSR)
(GFSR) edge [bend right=20] node[scale=.7] (qFSR) [below] {$-\circ F'_0$} (G'FSR)
	edge [bend left=20] node[scale=.7] (q'FSR) [above] {$-\circ F_0$} (G'FSR);

\draw[-{Implies},double distance=1.5pt,shorten >=10pt,shorten <=10pt] (q'FSR) to node[scale=.7] [right, xshift=0.2cm] {$-\circ p_0$} (qFSR);
\end{tikzpicture}
\end{center}
\newpage
\item \label{ax:map-lift} the following pasting diagrams are equal 
\begin{center}
\begin{tikzpicture}
\node (QGF) at (0,6) {$\catK(Y,-)$}; 
\node (QG'F) at (4,6) {$\catK(X,-)$};
\node (GTF) at (0,3) {$\RRM_T(-)$};
\node (G'TF) at (4,3) {$\RRM_S(-)$};

\path[->]
(QGF) edge [bend left=20] node[scale=.7] (Qq'F) [above] {$-\circ F$} (QG'F)
	edge node[scale=.7] [left] {$U^\ast_Y$} (GTF)
	edge [bend right=20] node[scale=.7] (QqF) [below] {$-\circ F'$} (QG'F)
(QG'F) edge node[scale=.7] [right] {$U^\ast_X$} (G'TF)
(GTF) edge [bend right=20] node[scale=.7] (q'TF) [below] {$\tilde{F'}$} (G'TF);

\draw[-{Implies},double distance=1.5pt,shorten >=10pt,shorten <=10pt] (Qq'F) to node[scale=.7] [right] {$-\circ p$} (QqF);
\draw[-{Implies},double distance=1.5pt,shorten >=25pt,shorten <=25pt] (QqF) to node[scale=.7] [right,xshift=0.2cm] {$\tilde{\phi'}$} (q'TF);

\node at (5.5,4.5) {$=$};

\node (GTFR) at (7,6) {$\catK(Y,-)$};
\node (G'TFR) at (11,6) {$\catK(X,-)$};
\node (GFSR) at (7,3) {$\RRM_T(-)$};
\node (G'FSR) at (11,3) {$\RRM_S(-)$.};

\path[->]
(GTFR) edge [bend left=20] node[scale=.7] (q'TFR) [above] {$-\circ F$} (G'TFR)
	edge node[scale=.7][left] {$U^\ast_Y$} (GFSR)
(G'TFR) edge node[scale=.7] [right] {$U^\ast_X$} (G'FSR)
(GFSR) edge [bend right=20] node[scale=.7] (qFSR) [below] {$\tilde{F'}$} (G'FSR)
	edge [bend left=20] node[scale=.7] (q'FSR) [above] {$\tilde{F}$} (G'FSR);

\draw[-{Implies},double distance=1.5pt,shorten >=10pt,shorten <=10pt] (q'FSR) to node[scale=.7] [right, xshift=0.2cm] {$\tilde{p}$} (qFSR);
\draw[-{Implies},double distance=1.5pt,shorten >=25pt,shorten <=25pt] (q'TFR) to node[scale=.7] [right, xshift=0.2cm] {$\tilde{\phi}$} (q'FSR);
\end{tikzpicture}
\end{center}
\end{enumerate}
\end{defn}

\begin{prop}
\label{prop:lift-iff-rel-2-cells}
Let $(F,\,F_0,\,\tilde{F},\,\tilde{\phi})$ and $(F',\,F'_0,\,\tilde{F'},\,\tilde{\phi'})$ be two liftings to relative right modules and $(F,\,F_0,\,\phi)$ and $(F',\,F'_0,\,\phi)$ their corresponding relative monad morphisms (using Proposition~\ref{prop:lift-iff-rel-1-cells}). A map of liftings to relative right modules between them is equivalent to a relative monad transformation between the corresponding relative monad maps. 
\end{prop}

\begin{proof}
Given $(p,\,p_0,\,\tilde{p})$ we can see that $(p,\,p_0)$ is a relative monad transformation. This follows from $J^\ast_{X_0}\tilde{p}=(-\circ p_0)J^\ast_{Y_0}$ and part~(\ref{ax:map-lift}) of Definition~\ref{defn:map-lift-rrm} applied to~$1_Y$. On the other hand, given a relative monad transformation $\tilde{p}$, to satisfy the first axiom of map of liftings to relative right modules, we need to choose $\tilde{p}$ as follows: for any $K\in\catK$ and $(M,\,(-)_M)\in\RRM_T(K)$, then $\tilde{p}_{K, M}\colon \tilde{F}(M,\,(-)_m)\rightarrow$~$\tilde{F'}(M,\,(-)_m)$ is defined as $Mp_0\colon (MF_0,\,(-)_{\tilde{F}m})\rightarrow(MF'_0,\,(-)_{\tilde{F'}m})$. Thanks to Proposition~\ref{prop:lift-iff-rel-1-cells} we know that for any 2-cells $f\colon IA\rightarrow SA$, $f_{\tilde{F}m}=(\phi B\cdot Ff)_m$ and $f_{\tilde{F'}m}=(\phi' B\cdot F'f)_m$. Therefore, to prove that $\tilde{p}_{K, M}$ is a map of relative right modules, it suffices to prove the following equality:
\begin{align*}
Mp_0B\cdot(\phi B\cdot Ff)_m
&  = (Tp_0B\phi B\cdot Ff)_m
& (\textrm{by right naturality of}\;(-)_m) 
\\
& = (\phi' B\cdot pSB \cdot Ff)_m
& (\textrm{because}\;(p,p_0)\in\Rel(\catK))
\\
& = (\phi' B\cdot F'f\cdot pIA)_m
& (\textrm{by naturality of}\;p)
\\
& = (\phi' B\cdot F'f\cdot Jp_0A)_m
& (\textrm{because}\;(p,p_0)\in\Rel(\catK))
\\
& = (\phi' B\cdot F'f)_m\cdot Mp_0A
& (\textrm{by left naturality of}\;(-)_m).
\end{align*} 
These constructions are clearly inverses of each other. 
\end{proof}

\begin{prop}
Let $\catK$ be a 2-category. Then there exists a 2-category $\RLift(\catK)$ of lifting to relative right modules with objects relative monads in $\catK$, 1-cells liftings to relative right modules and 2-cells maps between them. 
\end{prop}

\begin{proof}
The composition is given by composition in $\catK$ and pasting the appropriate diagrams. The strictness of this operation follows from the strictness in $\catK$ and the Pasting Theorem for 2-categories \cite{AJPow90}.
\end{proof}

Starting from the formal theory of monads \cite{StreetR:fortm}, one can prove that the 2-category of monads in a 2-category is equivalent to the 2-category of liftings to indexed algebras. The following Theorem provides a similar result in the setting of relative monads, which will be useful to prove \Cref{thm:rel-distr-eq-kl}. 

\begin{theorem}
\label{thm:rel-eq-rlift}
Let $\catK$ be a 2-category. Then the 2-categories $\Rel(\catK)$ and $\RLift(\catK)$ are 2-isomorphic. 
\end{theorem}

\begin{proof}
Let us define a 2-functor $\Gamma\colon \Rel(\catK)\rightarrow\RLift(\catK)$. On objects we take the identity, on 1-cells we use the correspondence seen in Proposition~\ref{prop:lift-iff-rel-1-cells} and on 2-cells the one seen in Proposition~\ref{prop:lift-iff-rel-2-cells}. These propositions prove also that $\Gamma$ is an isomorphism on hom-categories, and therefore a 2-isomorphism. 
\end{proof}

\subsection*{Beck's Theorem for Relative Distributive Laws}
\addcontentsline{toc}{subsection}{Beck's Theorem for Relative Distributive Laws}
\markboth{Beck's Theorem for Relative Distributive Laws}{Beck's Theorem for Relative Distributive Laws}

Throughout this section we will consider a relative monad $(X,\,I,\,T)\in\Rel(\catK)$ and monads $(X,\,S),\,(X_0,\,S_0)\in\Mnd(\catK)$ compatible with $I$. We will start by introducing the generalised notion of extensions to Kleisli categories, which we call \emph{lifting to relative right modules}. Then we will prove that this concept is equivalent to relative distributive laws, providing a Beck-type equivalence. 

\begin{defn}
\label{defn:ext-to-kl}
We define a \emph{lifting of $S$ to the relative right modules of $T$} as a monad $\tilde{S}\colon \RRM_T(-)\rightarrow\RRM_T(-)$ and a 2-cell $d^\ast\colon U^\ast\tilde{S}\to(-\circ S)U^\ast$ satisfying the following properties. 
\begin{enumerate}[(i)]
\item The pair $(\tilde{S},\,-\circ S_0)$ is a monad compatible with $J_0^\ast$.

\item  The pair $(U^\ast,d^\ast)\colon (\RRM_T(-),\,\tilde{S})\rightarrow(\catK(X,-),\,-\circ S)$ is a morphism in $\Mnd(\hat{\catK}^\op)^\op$.  

\item The modification $t\colon (-\circ I,\,1)\rightarrow(J_0^\ast,\,1)\circ(U^\ast,\,d^\ast)$ is a 2-cell in $\Mnd(\hat{\catK}^\op)^\op$.

\item \label{ext:new-ax} Let us denote with $(-)_{\tilde{S}m}$ and $(-)_{\tilde{S}T}$ the relative right module structure operators of $\tilde{S}(M,\,(-)_m)$ and $\tilde{S}(T,\,(-)^\dagger)$ respectively. Then, for any 2-cell $f\colon IA\rightarrow TB$ and any $K$-indexed relative right module $(M,\,(-)_m)$, the action of $\tilde{S}$ has to be
\begin{center}
$f_{\tilde{S}m}=(f_{\tilde{S}T}\cdot tS_0A)_m$.
\end{center}
\end{enumerate}
\end{defn}

\begin{eg}
\label{eg:lift-rrm-in-cat}
Let us consider the case $\catK=\Cat$. In this case, part (\ref{ext:new-ax}) of Definition\til\ref{defn:ext-to-kl} derives from the first three and properties of $\Cat$. More precisely, we recall that $f_m$ is equal to $\bar{M}f$, with $\bar{M}\colon \Kl(T)\rightarrow\mathbb{K}$ the functor associated to the relative right module $M$, and so $f_{\tilde{S}m}=\bar{M}\tilde{S}f$. On the other hand $f_{\tilde{S}T}=(\tilde{S}f)^\dagger$, therefore we get the equality in (iv). 

Let us denote with $\catC_0$ and $\catC$ the categories on which we take the monads $S_0$ and $S$, and relative monad $T$. With this notation, we can rewrite the definition above in the following equivalent way. Let $\compS$ be a monad compatible with $I\colon \catC_0\rightarrow \catC$ and $(\catC,\,I,\,T)$ a relative monad in $\Cat$. We denote with $J_0\colon\catC_0\to\Kl(T)$ and $U\colon\Kl(T)\to\catC$ the functors forming the Kleisli relative adjunction (see \cite[Section 2.3]{AltChap:mnd-no-end}). We define an \emph{extension of $S$ to the Kleisli category of $T$} as a monad $\tilde{S}\colon \textrm{Kl}(T)\rightarrow\textrm{Kl}(T)$ such that:
\begin{enumerate}[(i)]
\item $\tilde{S}J_0=J_0S_0$ and $(J_0,\,1)$ becomes a monad morphism, i.e.\ $\tilde{m}J_0=J_0m_0$ and $\tilde{s}J_0=J_0s_0$;
\item  the functor $U\colon \textrm{Kl}(T)\rightarrow\catC$ is a monad morphism (with 2-cell $d$);
\item the unit $t\colon (I,\,1)\rightarrow(U,\,d)\circ(J_0,\,1)$ is a monad transformation. 
\end{enumerate}
\end{eg}

\begin{rmk}
More generally, if $\catK$ has relative Kleisli objects $X_{I,T}$, then we can rewrite Definition~\ref{defn:ext-to-kl} as a particular extension. First, let us notice that, if $\RRM_T(-)$ is represented by $X_{I,T}$, then there is a universal relative right module $(J_0,\,(-)_{j_0})$ with $J_0\colon X_0\to X_{I,T}$. 
Therefore, diagram \ref{diag:rrm} becomes equivalent (using Yoneda for 2-categories) to
\begin{center}
\begin{tikzpicture}
\node (a) at (0,0) {$X_0$};
\node (b) at (2,1.5) {$X_{I,T}$};
\node (c) at (4,0) {$X$,};

\path[->] 
(a) edge node[scale=.7] [left,xshift=-0.1cm,yshift=0.1cm]{$J_0$} (b)
	edge node[scale=.7] (I) [below]{$I$} (c)
(b) edge node[scale=.7] [right, xshift=0.1cm,yshift=0.1cm]{$U$} (c)
;

\draw[-{Implies},double distance=1.5pt,shorten >=10pt,shorten <=10pt] (I) to node[scale=.7] [right, xshift=0.1cm] {$t$} (b);
\end{tikzpicture}
\end{center}
with $UJ_0=T$. Therefore, a lifting of a monad $S$ to relative right modules of $T$ becomes equivalent to an \emph{extension of $S$ to $X_{I,T}$}, i.e.\ a monad $\tilde{S}\colon X_{I,T}\to X_{I,T}$ such that 
\begin{enumerate}[(i)]
\item the pair $(\tilde{S},S_0)$ is a monad compatible with $J_0$;

\item  the 1-cell $U\colon (X_{I,T},\,\tilde{S})\rightarrow(X,\,S)$ is a monad morphism (with 2-cell $d$);

\item the 2-cell $t\colon (I,\,1)\rightarrow(U,\,d)\circ (J_0,\,1)$ is a monad transformation;

\item for any 2-cell $f\colon IA\rightarrow TB$ we have the following equality 
\begin{center}
$\tilde{S}f_{j_0}=(U\tilde{S}f_{j_0}\cdot tS_0A)_{j_0}$.
\end{center}
\end{enumerate}
\end{rmk}

\begin{theorem}
\label{thm:rel-distr-eq-kl}
Let $(X,\,I,\,T)$ be a relative monad and $(S,\,S_0)$ a monad compatible with $I$, both in $\catK$. Then, relative distributive laws $d\colon ST\rightarrow TS_0$ are equivalent to liftings of $S$ to the relative right modules of $T$. 
\end{theorem}

\begin{proof}
We have already seen in Proposition~\ref{prop:rel-distr-mnd(rel)} that relative distributive laws are the objects of $\Mnd(\Rel(\catK))$. Now, using Theorem~\ref{thm:rel-eq-rlift}, we get that 
\begin{center}
$\Mnd(\Rel(\catK))\cong\Mnd(\RLift(\catK))$.
\end{center}
To get the conclusion, we just need to notice that an extension of $S$ to the \yellow relative right modules \black of $T$ is just an object of $\Mnd(\RLift(\catK))$.
\end{proof}

Putting together Theorems~\ref{thm:rel-distr-eq-lift} and~\ref{thm:rel-distr-eq-kl} we get the following Theorem, which gives us the counterpart of Beck's Theorem for relative distributive laws. 

\begin{theorem}
\label{thm:rel-beck}
Let $\catK$ be a 2-category, $(X,\,I,\,T)$ a relative monad in $\catK$ and $\compS$ a monad compatible with $I$. The following are equivalent: 
\begin{enumerate}[(i)]
\item a relative distributive law of $T$ over $\compS$;

\item a lifting $\widehat{T}\colon \Soalg(-)\longrightarrow\Salg(-)$ of $T$ to the algebras of $\compS$;

\item a lifting $\tilde{S}\colon \RRM_T(-)\rightarrow\RRM_T(-)$ of $S$ to the relative right modules of $T$. 
\end{enumerate}
\end{theorem}

Let us consider the particular case $\catK=\Cat$. We know that $\Cat$ has both relative EM objects and relative Kleisli objects (Examples~\ref{eg:EM-obj-cat} and~\ref{eg:Kl-obj-cat}). Using this property of $\Cat$ and Example~\ref{eg:lift-rrm-in-cat} we can show that Theorem~\ref{thm:rel-beck} can be rephrased in the following way.

\begin{cor}
Let $(\catC,\,I,\,T)$ be a relative monad in $\Cat$ and $\compS$ a monad compatible with $I$. The following are equivalent: 
\begin{enumerate}[(i)]
\item a relative distributive law of $T$ over $\compS$;

\item a lifting $\widehat{T}\colon \Soalg\rightarrow\Salg$ of $T$ to the algebras of $\compS$;

\item an extension $\tilde{S}\colon \Kl(T)\rightarrow\Kl(T)$ of $S$ to the Kleisli category of $T$. 
\end{enumerate}
\end{cor}


We conclude the chapter with a pair of examples of relative distributive laws in the 2-category of locally small categories $\Cat$. 

\begin{eg}[Power set and free monoids]
\label{eg:pow-set-free-mon}
We will consider a variation of a distributive law between two monads. Recall that there exists a distributive law between the power set monad $P$ and the free monoid one $S^M$, given by
\begin{center}
\begin{tikzpicture}
\node (a) at (0,0.5) {$d_X\colon S^MPX$};
\node (b) at (3.5,0.5) {$PS^MX$};
\node (c) at (0.5,0) {$I_1...I_n$};
\node (d) at (3.5,0) {$\lbrace a_1...a_n\mid a_i\in I_i\rbrace$.};

\draw[->] (a) to node[scale=.7] [above]{} (b);
\draw[|->] (c) to node[scale=.7] [above]{} (d);
\end{tikzpicture}
\end{center}
A problem arises if we want to impose some restrictions on the cardinality of sets. For example, given any infinite cardinal $\kappa$, let $\Set_{\leq \kappa}$ be the category of sets with cardinality less or equal to $\kappa$. Then, the restriction of $P$ to $\Set_{\leq\kappa}$ it is not an endofunctor any more. Nevertheless, we can recover its monad-like structure considering it as a relative monad on the inclusion $I\colon \Set_{\leq\kappa}\rightarrow\Set$. More precisely, we can take as unit $t_X(x):=\lbrace x\rbrace$ (for any $X\in\Set$ and $x\in X$) and as extension of $f\colon X\rightarrow PY$ the map
\begin{center}
\begin{tikzpicture}
\node (a) at (0,0.5) {$f^\dagger\colon PX$};
\node (b) at (2.3,0.5) {$PY$};
\node (c) at (0.4,0) {$I$};
\node (d) at (2.5,0) {$\bigcup_{i\in I}f(i)$.};

\draw[->] (a) to node[scale=.7] [above]{} (b);
\draw[|->] (c) to node[scale=.7] [above]{} (d);
\end{tikzpicture}
\end{center}

Let us consider now the restriction of $S^M$ to $\Set_{\leq\kappa}$. Let $X$ be a set of cardinality at most $k$, then $S^MX$ has cardinality
$$|S^MX|=|\coprod_{n\in\mathbb{N}}|X|^n\,|\;\leq\;|\coprod_{n\in\mathbb{N}}\kappa^n\,|=|\coprod_{n\in\mathbb{N}}\kappa|=\kappa$$
and therefore we get $S^M_\kappa\colon \Set_{\leq\kappa}\rightarrow\Set_{\leq\kappa}$. In particular it means that $(S^M,\,S^M_\kappa)$ is compatible with $I\colon \Set_{\leq\kappa}\rightarrow\Set$. With this point of view, we see that $d$ becomes a relative distributive law of $T:=P_\kappa\colon \Set_{\leq\kappa}\rightarrow\Set$ over $(S^M,\,S^M_\kappa)$. 

A similar situation arises when one works in set theories that do not have the power-set axiom, like Kripke-Platek set theory \cite{Barw:adm-sets} or Constructive Zermelo-Frankel set theory \cite{Acz:constr-set-th}. There, the power set operation can be viewed as a relative monad over the inclusion of
the category of sets into the category of classes. Note also that the presheaf construction can then be viewed as a categorified version of the power-set monad \cite{FioreM:relpsm}. See also \cite[Section~3.2.2]{HylaM:gen-domain-th}. 

\end{eg}

\yellow
\begin{rmk}[/Example]
\Cref{eg:pow-set-free-mon} is a particular case of a more general situation. Let us consider $\catC$ a category, $S,T\colon\catC\to\catC$ two monads and $d\colon ST\to TS$ a distributive law of $T$ on $S$. Then, given any subcategory $I\colon\catC_0\hookrightarrow\catC$, the monad structure of $T$ induces a relative monad $T_I$ on $I$. Moreover, if for any $x\in\catC_0$, $Sx$, $m_x$ and $(s_0)_x$ are all in $\catC_0$, then $d$ restricts to a relative distributive law of $T_I$ over $(S,S_0)$. In \Cref{eg:pow-set-free-mon} the subcategory $I\colon \Set_{\leq\kappa}\rightarrow\Set$ is full, so we only have to check that the functor $S^M$ restricts to $\Set_{\leq\kappa}$. 
\end{rmk}
\black 

\begin{eg}[Pointed vector spaces] In \cite[Example~1.1]{AltChap:mnd-no-end} is presented the relative monad $V$ of vector spaces. In order to define it, let us fix a semiring~$R$. For any set $X$, we will denote with $\delta_x\colon X\rightarrow R$ the map sending $x$ to 1 and everything else to 0. Then, $V$ is defined on the inclusion $I\colon \Fin\rightarrow\Set$ of finite cardinals into sets as follows: 
\begin{itemize}
\item for any finite cardinal $n$, $Vn:=\Set(In,\,R)$;
\item the unit $v_n\colon In\rightarrow Vn$ is defined, for any $i\in In$, as $v_n(i):=\delta_i$; 
\item given $\alpha\colon In\rightarrow Vm$ we define its extension $\alpha^\dagger\colon Vn\rightarrow Vm$ as, for any $f\colon In\to R$, 
$$\alpha^\dagger(f):=\sum_{i\in n}f(i)\cdot\alpha(i)(-)\colon Im\rightarrow R.$$
\end{itemize}
Let us consider the monad of pointed set, i.e.\ $SX:=X+1$, the unit $s_X\colon X\rightarrow X+1$ is the canonical inclusion and the multiplication $m_X\colon X+2\rightarrow X+1$ fixes any element of $X$ and sends the two elements of 2 to the only one of 1. One can easily prove that the category of $S$-algebras is the category of pointed sets. 

Clearly, defining $S_f$ as the restriction of $S$ to finite cardinals, we can see $(S,\,S_f)$ as a monad compatible with $I$. Moreover, there is a lifting of $V$ to the algebras of $(S,\,S_f)$ defined as follows: for any finite pointed set $(n,\,i)$ we set $\widehat{V}(n,\,i):=(\Set(In,\,R),\,\delta_i)$. Now we need to check that both the unit and the extension operator lift. First of all, we can see straight away that $v_n$ is a map of pointed sets, since by definition it sends each $i$ to the map $\delta_i$. Then, if we consider a map of pointed sets $\alpha\colon (In,\,i)\rightarrow (Vm,\,\delta_j)$ we need to check that the extension $\alpha^\dagger\colon Vn\rightarrow Vm$ is still a map of pointed sets, i.e.\ the equality $\alpha^\dagger(\delta_i)=\delta_j$ holds. For any $s\in m$
\begin{align*}
[\alpha^\dagger(\delta_i)](s)
&  = \sum_{r\in n}\delta_i(r)\cdot[\alpha(r)](s)
& (\textrm{by definition of}\;(-)^\dagger) 
\\
& = \alpha(i)(s)
& (\textrm{by definition of}\;\delta_i)
\\
& = \delta_j(s)
& (\alpha\;\textrm{map of pointed sets}).
\end{align*} 
Therefore, by Theorem\til\ref{thm:rel-distr-eq-lift}, we have a relative distributive law of $V$ over $(S,\,S_f)$. In particular, by Theorem\til\ref{thm:rel-distr-eq-kl}, we have a monad $\tilde{S}$ induced on the Kleisli of~$V$, i.e.\ vector spaces. What we get as algebras over this monad are \emph{pointed vector spaces}. 
\end{eg}

%
%
%
%

\bibliography{References}
\bibliographystyle{plain}

\end{document}